\documentclass[a4paper, 11pt, reqno]{amsart}
\usepackage{graphicx}
\usepackage{amsmath}
\usepackage{amssymb}
\usepackage{amsfonts}
\usepackage{verbatim}
\usepackage{scalefnt}
\usepackage{multirow}
\usepackage{enumerate}
\usepackage{mathtools}
\usepackage{float}
\usepackage{enumitem}
\restylefloat{figure}
\usepackage[margin=2.9cm]{geometry}

\usepackage{fancyhdr}
\usepackage{hyperref}

\newcommand{\COMMENT}[1]{}

\newcommand{\eps}{\varepsilon}

\renewcommand{\restriction}{\mathord{\upharpoonright}}

\newtheorem{theorem}{Theorem}
\newtheorem{lemma}[theorem]{Lemma}
\newtheorem{proposition}[theorem]{Proposition}

\newtheorem{claim}[theorem]{Claim}

\newtheorem*{definition*}{Definition}

\numberwithin{equation}{section}
\numberwithin{theorem}{section}

\title[Optimal bounds on the polynomial Schur's theorem]{Optimal bounds on the polynomial Schur's theorem}

\author{Jaehoon Kim}
\email{jaehoon.kim@kaist.ac.kr}
\address{Department of Mathematical Sciences, KAIST, Daejeon 34141, Republic of Korea}
\author{Hong Liu}
\email{hongliu@ibs.re.kr}
\address{Extremal Combinatorics and Probability Group (ECOPRO), Institute for Basic Science (IBS), Daejeon, South Korea.}

\author{P\'eter P\'al Pach}
\email{pach.peter@vik.bme.hu}
\address{Department of Computer Science and Information Theory, Budapest University of Technology and Economics, M\H{u}egyetem rkp. 3., H-1111 Budapest, Hungary; \newline \hspace*{4mm}
MTA-BME Lend\"ulet Arithmetic Combinatorics Research Group,
  ELKH, M\H{u}egyetem rkp. 3., H-1111 Budapest, Hungary.}

\thanks{J.K. was supported by the Fulbright Visiting Scholar Fellowship and by the National Research Foundation of Korea (NRF) grant funded by the Korean government(MSIT) No. RS-2023-00210430. H.L. was supported by IBS-R029-C4. P.P.P. was supported by the Lend\"ulet program of the Hungarian Academy of Sciences (MTA) and by the National Research, Development and Innovation Office NKFIH (Grant Nr. K146387).}
\begin{document}
\begin{abstract}
Liu, Pach and S\'andor recently characterized all polynomials $p(z)$ such that the equation $x+y=p(z)$ is $2$-Ramsey, that is, any $2$-coloring of $\mathbb{N}$ contains infinitely many monochromatic solutions for $x+y=p(z)$. 
In this paper, we find asymptotically tight bounds for the following two quantitative questions.
\begin{itemize}
    \item For $n\in \mathbb{N}$, what is the longest interval $[n,f(n)]$ of natural numbers which admits a $2$-coloring with no monochromatic solutions of $x+y=p(z)$?
    \item For $n\in \mathbb{N}$ and a $2$-coloring of the first $n$ integers $[n]$, 
    what is the smallest possible number $g(n)$ of monochromatic solutions of $x+y=p(z)$?
\end{itemize}
Our theorems determine $f(n)$ up to a multiplicative constant $2+o(1)$, and determine the asymptotics for $g(n)$.
\end{abstract}

\date{\today}
\maketitle
\section{Introduction}
\COMMENT{All `COMMENT's are intended to be deleted before submission, while `footnotes' are intended to be remained.}

Ramsey theory is a study about finding largest possible orders in chaotic systems.
It has a long history dating back to the following remarkable theorem by Schur~\cite{Schur1916} in 1916: the equation $x+y=z$ (or equivalently the tuple of the form $(x,y,x+y)$) is \emph{Ramsey}, which means that any finite coloring of $\mathbb{N}$ contains infinitely many tuples $(x,y,z)$ satisfying the equation $x+y=z$ (or equivalently $(x,y,z)=(x,y,x+y)$). These tuples satisfying the equation are called \emph{monochromatic solutions} of the equation. 
Another classical theorem regarding Ramsey equation is van der Waerden's theorem~\cite{van der Waerden} which states that the arithmetic progression $(x,x+d,\dots, x+(t-1)d)$ is Ramsey for any $d\in \mathbb{N}$.
Later, Rado~\cite{Rado} completely determined all Ramsey linear equations.
Starting with these results, many extensions have been considered.
Just to state a few, the Phytagorean equation $x^2+y^2=z^2$ was shown to be 2-Ramsey by Heule-Kullmann-Marek~\cite{HKM} using computer assistance, 
a polynomial extension of van der Waerden's theorem was shown by Bergelson-Leibman~\cite{BL},  Moreira~\cite{Moreira} proved that $(x,x+y,xy)$ is Ramsey, and Bowen-Sabok~\cite{BS} proved that $(x,y,x+y,xy)$ is Ramsey in the rationals.

Recently, there have been exciting developments on the polynomial extension of Schur's theorem.
Csikv\'ari-Gyarmati-S\'ark\"{o}zy \cite{CGS} proved that $x+y=z^2$ is not $16$-Ramsey, meaning that there exists a coloring of $\mathbb{N}$ with $16$ colors containing no monochromatic solutions for $x+y=z^2$ other than the trivial solution $(x,y,z)=(2,2,2)$. 
Green-Lindqvist \cite{GL} used a Fourier analytic argument to show that the equation $x+y=z^2$ is $2$-Ramsey but not $3$-Ramsey. 
The third author of this paper \cite{P} found a shorter combinatorial proof for the $2$-Ramseyness of the equation $x+y=z^2$.
The second and third authors of this paper together with S\'andor \cite{LPS} proved the polynomial Schur's theorem which completely determines the Ramseyness of the equation $x+y=p(z)$ for all polynomials $p$. The equations of the form $ax+by=p(z)$ for certain choices of $a,b,p(z)$ were further considered in  \cite{BDKPP}.

A modification of the $3$-coloring given by Green-Lindqvist \cite{GL} shows that $x+y=p(z)$ is not $3$-Ramsey for any polynomial $p(z)$ of degree at least two, and the following theorem determines all polynomials $p$ for which the equation $x+y=p(z)$ is $2$-Ramsey.
\begin{theorem}[Polynomial Schur's theorem \cite{LPS}]\label{thm: poly Schur}
    For a polynomial $p(z)$ of degree $d\geq 1$ with a positive leading coefficient, the equation $x+y=p(z)$ is $2$-Ramsey if and only if $p(z)$ is a non-odd polynomial. 
\end{theorem}
Here, a polynomial is \emph{odd} if it only attains odd values on $\mathbb{N}$ and \emph{non-odd} otherwise. It is easy to see that if the polynomial $p(z)$ is odd, then coloring all integers according to their parity will avoid any monochromatic solutions. Theorem~\ref{thm: poly Schur} shows that this `parity constraint' is the only obstruction for the $2$-Ramseyness of the equation $x+y=p(z)$. Note that it is sufficient to consider the polynomials with a positive leading coefficient as we only consider colorings of the positive integers.

As this result completely settles the Ramsey problem for the equation $x+y=p(z)$, it is natural to ask its quantitative aspect. 
In Ramsey theory, finding the correct quantitative bounds
tends to be much more challenging than just proving the existence. For example, in the (hyper)graph Ramsey theory, the existence of the Ramsey number was completely resolved \cite{Ramsey} already in 1930. However, the correct quantitative bounds still remain elusive.
Even for the simplest $2$-uniform (graph) Ramsey numbers, there is an exponential gap between the best known lower bound \cite{Spencer} and the best known upper bound and improving this exponential gap to a smaller one in the recent breakthrough by Campos-Griffiths-Morris-Sahasrabudhe \cite{CGMS} took almost 90 years.

Regarding the polynomial extension of Schur's theorem, one can ask two natural quantitative questions.
How long can we avoid the monochromatic solutions and how many monochromatic solutions we must encounter in a $2$-coloring of $[n]:=\{1,2,\ldots,n\}$?
For the first question, considering colorings from $1$ to a number $m$ avoiding monochromatic solutions makes little sense as such a number $m$ will always be a finite number. Hence, we should rather look at colorings of numbers that are at least as large as some number $n$. Formally, we study the following two quantitative questions regarding the $2$-Ramseyness of the equation $x+y=p(z)$ for non-odd polynomials $p(z)$.
\begin{itemize}
    \item[(1)] For $n\in \mathbb{N}$, what is the largest $f(n)$ such that one can 2-color the interval $[n,f(n)]$ so that it does not contain any monochromatic solutions of $x+y=p(z)$?
\item[(2)] What is the smallest possible number $g(n)$ of monochromatic solutions  of $x+y=p(z)$ inside the set $[n]$ over all $2$-colorings of $[n]$?
\end{itemize}

The second question of counting monochromatic solutions have been considered for linear equations.
Frankl, Graham and R\"odl~\cite{FGR} in the 80s proved an optimal quantitative Rado's theorem. More precise counting result for the specific equation $x+y=z$ was later obtained independently by Schoen~\cite{Schoen} and Robertson and Zeilberger~\cite{RZ}, another proof was given by Datskovsky~\cite{Dat}. 
Results on counting monochromatic solutions of systems of linear equations on more general groups than $\mathbb{Z}$ were also obtained, see \cite{CCS,SV}. However, for non-linear equations, only a few counting results are known for very specific types of equations. For example, diagonal Diophantine equations were considered using the Hardy–Littlewood circle method, see e.g. \cite{Prendiville}.
In general, determining the answers of these quantitative questions for a fairly substantial class of non-linear equations is very difficult.

For equation $x+y=p(z)$ of our interest, there are constructions yielding natural bounds. 
For the first question, this natural bound is $f(n)= \Omega(n^{d^2})$. 
Assuming that both numbers $p(n)$ and $p(\frac{1}{2}p(n))$ are even numbers, consider the coloring of the interval $[n, \frac{1}{2} p( \frac{1}{2}p(n) ) -1]$ as follows.
We color the numbers in the first interval $[n,   \frac{p(n)}{2}-1]$ by the color $+1$ and the rest of the numbers in $[ \frac{p(n)}{2} , \frac{1}{2} p(\frac{1}{2}p(n) ) -1]$ by $-1$. Then any $x,y, z$ with color $+1$ satisfies $p(z)\geq p(n)> x+y$ while any $x,y,z$ with color $-1$ satisfies $p(z)\geq p( \frac{1}{2}p(n))>x+y$, thus no monochromatic solutions exist. This shows that $f(n) \geq \frac{1}{2} p(  \frac{1}{2}p(n) )-1 = \Omega(n^{d^2})$.

Regarding the second question for the equation $x+y=p(z)$, the following construction shows that $g(n)= O(n^{2/d^2})$. Given $n$, consider two numbers $a = \min\{ x: p(x)> 2n\}$ and $b=\min\{ x: p(x)> 2a\}$ satisfying $a = \Theta(n^{1/d})$ and $b= \Theta(n^{1/d^2})$.
By coloring the numbers $[b-1]\cup [a,n]$ with color $+1$ and $[b,a-1]$ with color $-1$, we can avoid all monochromatic solutions $(x,y,z)$ with $x,y \geq b$. As all monochromatic solution $(x,y,z)$ lies in $[b-1]$, the number of monochromatic solutions is at most $b^2 = O(n^{2/d^2})$, showing that $g(n)= O(n^{2/d^2})$. 

In this paper, we almost completely resolve both questions. The first theorem below answers the first question by determining the value of $f(n)$ up to a multiplicative constant $2+o(1)$.\COMMENT{In fact, by improving some easy calculations one can show $f(n) \leq  p( \frac{p(n)}{2} ) - \frac{p(n)}{2}$.}

\begin{theorem}\label{thm: existence}
Let $p(z)\in \mathbb{Z}[z]$ be a non-odd polynomial of degree $d\geq 2$ with a positive leading coefficient, then there exists an integer $n_0$ such that the following holds for all $n\geq n_0$.
For any $2$-coloring $\phi$ of $[n, p(\lceil \frac{1}{2}p(n) \rceil)]$, there exists monochromatic numbers $x,y,z\in [n, p(\lceil \frac{1}{2}p(n)\rceil)]$ satisfying $x+y=p(z)$.
\end{theorem}
Our second theorem determines the asymptotics for the answer of the second question, showing that the number $2/d^2$ in the exponent is tight. %This provides optimal quantitative results for the Ramseyness of a fairly substantial class of non-linear equations.

\begin{theorem}\label{thm: counting}
Let $p(z)\in \mathbb{Z}[z]$ be a non-odd polynomial of degree $d\geq 2$ with a positive leading coefficient. Let $\phi$ be a $2$-coloring of $[n]$. 
Then there are at least $n^{2/d^2-o(1)}$ monochromatic solutions $(x,y,z)\in [n]^{(3)}$ of $x+y=p(z)$. 
\end{theorem}

By closely examining our proof, one can measure the above $o(1)$-term more precisely. Theorem~\ref{thm: counting} yields $\Omega(\frac{n^{2/d^2}}{ \log (n)})$ monochromatic solutions if $d\geq 3$, and $n^{\frac{2}{d^2}- O(\frac{1}{\log\log n})}$ monochromatic solutions if $d=2$. We note that the multiplicative $\Omega(\frac{1}{\log n})$-term for $d\geq 3$ comes from Lemma~\ref{lem: many intervals} and the multiplicative $n^{-O(\frac{1}{\log\log n})}$-term for $d=2$ comes from Lemma~\ref{lem: non isolated}.

\section{Preliminaries}

For given natural numbers $n,a,b\in \mathbb{N}$, we define $[n]:=\{1,2,\dots, n\}$ and $[a,b] := \{a, a+1,\dots, b\}$. 
  We write $a=b\pm c$ if $b-c \leq a \leq b+c$. All logarithms in this paper are the logarithms with the base two.
If we claim that a result holds whenever we have $0<a\ll b\ll c,d<1$, it means that there exists functions $f,g$ such that the result holds as long as $a<f(b)$ and $b<g(c,d)$. We will not compute these functions explicitly.
In many cases, we treat large numbers as if they are integers by omitting floor and ceiling if it does not affect the argument.
%We write $x=O_{a,b}(y)$ if there exists a function $f$ such that $x\leq f(a,b)y$. Similarly, we define $\Omega_{a,b}$ and $\Theta_{a,b}$.
For a number $c$ and sets $A, B\subseteq \mathbb{N}$ of numbers, we write $A+c = \{a+c : a\in A\}$, $c-A=\{ c-a: a\in A\}$, $A-c=\{a-c:a\in A\}$,  $cA = \{ca : a\in A\}$, $A+B= \{a+b: a\in A, b\in B\}$ and $A-B= \{a-b: a\in A, b\in B\}$.
Furthermore, if $A$ is a finite set of numbers, we write $\min(A)$ and $\max(A)$ to denote the minimum and the maximum number in the set $A$.
For a given sequence $a_1,\dots, a_t$, we say that a number $Q$ is an \emph{integral combination} of the sequence $a_i$'s if  $Q=\sum_{i\in [t]} b_i a_i$ for some integers $b_1,\dots, b_t\in \mathbb{Z}$. If all $b_i$'s are non-negative, then $Q$ is a \emph{non-negative integral combination} of the sequence of $a_i$'s.

We say that a polynomial $p(z)\in \mathbb{Z}[z]$ is \emph{odd} if $p(z)$ is odd for all $z\in \mathbb{Z}$ and \emph{non-odd} otherwise. Throughout the paper, we fix a non-odd polynomial 
$$p(z) = a_d z^d + \dots +a_1 z+ a_0 \in \mathbb{Z}[z]$$ 
of degree $d\geq 2$ with positive leading coefficient $a_d$. We write $\|p\|$ to denote $d+\sum_{i=0}^{d} |a_i|$. 
We also assume that a $2$-coloring $\phi: \mathbb{N}\rightarrow \{-1,1\}$ is given. We say that a triple $(x,y,z)$ is a \emph{monochromatic solution} if it satisfies $x+y=p(z)$ and $x,y,z$ all have the same color. We consider $(x,y,z)$ and $(y,x,z)$ as two distinct solutions.
For two functions  $f,g$, we write $f(x)=O(g(x))$ if $|f(x)|\leq C|g(x)|$ for some $C$ which depends only on $\|p\|$. We also define the expression $f(x)=\Theta(g(x))$ and $f(x)=\Omega(g(x))$ in a similar way.

For a given $2$-coloring $\phi$, Let $\Phi^+, \Phi^-$ denote the set of natural numbers with colors $+1$ and $-1$, respectively. 
We say that a number $k$ is a \emph{switch} if $\phi(k)\neq \phi(k+1)$, and it is a \emph{positive switch} if $\phi(k)=+1$ and $\phi(k+1)=-1$. Suppose that we have fixed a large number $k$ with $1/k\ll 1/\|p\|$ which is a positive switch. We may assume this as otherwise either we obtain many monochromatic solutions or we can swap the colors $+1$ and $-1$ to obtain such a $k$. 

For each $s\in \mathbb{N}$, we write 
$$m_s(k):= p(k+s)- p(k).$$
As $1/k\ll 1/\|p\|$ holds, we have $m_{s}(k) >0$.
We consider the following two important numbers:
$$ k_1 := \max\Big\{ t: p(t)< \frac{1}{2}p(k)\Big\} \enspace \text{and} \enspace k_{0} := \max\{ t: p(t) < 2p(k) - 4m_1(k)\}.$$ 
With this definition, we can estimate the value of $k_1,k_0$ as $
k_1 = (2^{-1/d} \pm o(1) )k \text{ and } k_{0} = (2^{1/d} \pm o(1) ) k$, where the $o(1)$ term tends to zero as $k$ grows. 
We say that the switch $k$ is \emph{isolated} if all numbers in $[k+1,k_{0}]$ has color $-1$, otherwise it is \emph{non-isolated}.\footnote{
Note that this definition is slightly different from \cite{LPS} as we do not care how $[k/2,k]$ looks like.
} 

 We say that a pair $\{a,b\}$ of two numbers is \emph{$(k,s)$-bad} (or simply  
\emph{bad}) for $s\in [k]$ if 
$$b=a+m_s(k), \enspace \phi(k+s)=-1, \enspace \text{and} \enspace \phi(a) > \phi(b).$$
Since $\phi$ only takes values in $\{-1,1\}$, the above inequality $\phi(a)>\phi(b)$ fixes the colors of $a$ and $b$ as $\phi(a)=+1$ and $\phi(b)=-1$.
This concept of bad pair plays an important role in this paper because it provides many monochromatic solutions as follows.

\begin{lemma}\label{lem: bad}
Suppose that $p(z)\in \mathbb{Z}[z]$ is a polynomial with positive leading coefficient of degree $d\geq 2$, and $\phi:\mathbb{N}\rightarrow \{-1,1\}$. Suppose that $k$ is a positive switch with $0<1/k\ll 1/\|p\|< 1$.
If we have $t$ pairwise disjoint bad sets, then we obtain at least $t$ distinct monochromatic solutions.
\end{lemma}
\begin{proof}
Assume that $\{a,b\}$ is a $(k,s)$-bad set with $b=a+m_s(k)$, then $\phi(a)=+1$ and $\phi(b)=-1$.
As $1/k\ll 1/\|p\|$, we know that $m_s(k)>0$.

If $p(k)-a$ is colored $+1$, then $(a,p(k)-a,k)$ forms a monochromatic solution with color $+1$.
If $p(k)-a$ is colored $-1$, then $(b,p(k)-a,k+s) = (b, p(k+s)-b,k+s)$ forms a monochromatic solution with color $-1$. Hence, we obtain a monochromatic solution of the form $(a,p(k)-a,k)$ or $(b,p(k)-a,k+s)$.

Assume that we have pairwise disjoint sets $\{a_1,b_1\},\dots, \{a_t, b_t\}$ such that each $\{a_i,b_i\}$ is a $(k,s_i)$-bad set.
For each $i\in [t]$, the above argument shows that either $(a_i,p(k)-a_i,k)$ or $(b_i, p(k)-a_i,k+s_i)$ is a monochromatic solution. 
As two solutions $(x,y,z)$ and $(y,x,z)$ are different and $\{a_1,b_1\},\dots, \{a_t,b_t\}$ are all pairwise distinct, those solutions are all distinct and provide $t$ distinct monochromatic solutions.
\end{proof}

\subsection{Proof ideas}\label{sec: proof idea}
In this section, we describe the ideas for the proofs of our theorems. 
Assume that our task is to find many monochromatic solutions. 
As Lemma~\ref{lem: bad} demonstrates, this task boils down to finding many pairwise disjoint bad pairs. 
Using a lemma in \cite{LPS} (see Lemma~\ref{lem: non isolated}), we can assume that there is a number $k$ which is an isolated switch, i.e. $[k+1,k_0]\subseteq \Phi^{-}$.
Exploiting this assumption, we will pursue the following strategy.

As $[k+1,k_0]\subseteq \Phi^{-}$, for some number $z\in \Phi^{-}$, the interval $p(z) - [k+1,k_0]$ mostly consists of numbers with color $+1$, as otherwise we obtain many monochromatic solutions $(x,p(z)-x,z)$ with $x\in [k+1,k_0]$. Similar argument yields that for $z'\in \Phi^{+}$, $p(z') - (p(z)-[k+1,k_0])$ mostly consists of the numbers with the color $-1$.
Exploiting these arguments, we will be able to find intervals $P$ of mostly positive colored numbers and intervals $N$ of mostly negative colored numbers satisfying $\max(P)< \min(N)$. This provides many pairs $(a,b)\in P\times N$ satisfying $a< b$, $\phi(a)=+1$ and $\phi(b)=-1$. 
Based on this choice, we will write the number $b-a$ as a non-negative integral combinations of the numbers $m_1(k),\dots, m_{k_0-k}(k)$.
For example, if $b-a = 2m_1(k) + 2 m_2(k) $, then we may consider a sequence 
$$a,~ a+m_1(k),~ a+m_1(k)+m_2(k),~ a+m_1(k)+2m_2(k),~ a+ 2m_1(k)+ 2m_2(k) = b.$$
As the $\phi$-value on this sequence starts with $+1$ and ends with $-1$ and $k+s \in \Phi^-$ for $s\in [2]$, there must be two consecutive numbers $x,y$ in the sequence with $\phi(x)>\phi(y)$ and $y= x+ m_{s}(k)$ for some $s\in [2]$. This provides a bad pair $\{x,y\}$.
This strategy suggests that it is important to determine the numbers $Q=b-a$ that can be written as a non-negative integral linear combination of $m_{s}(k)$.

Suppose that the above procedure yields a bad pair $\{x,y\}= \{a+m_1(k), a+m_1(k)+m_2(k)\}$, and we wish to find more bad pairs disjoint from $\{x,y\}$. In order to obtain such bad pairs, we consider another sequence 
$$a,~ a+m_2(k),~ a+2m_2(k),~ a+ m_1(k)+2 m_2(k),~ a+2m_1(k)+2m_2(k) = b.$$
Note that this sequence lacks of consecutive numbers of the form $a+m_1(k)$ and $a+m_1(k)+m_2(k)$. As before, this sequence yields a new bad pair. However, whichever new bad pair we obtain, it has a different form than $\{a+m_1(k), a+m_1(k)+m_2(k)\}$ due to the choice of this new sequence. 
In general, having different forms does not necessarily imply disjointness as one number can be expressed as two different non-negative integral combinations of $m_s(k)$.
Nonetheless, we can show that the numbers $m_s(k)$ admit some `independence' (see Proposition~\ref{prop: t terms independent}). Using this independence and expressing the number $b-a$ as integral combinations of only a restricted number of terms $m_{s}(k)$, we can prove that the new bad pair is indeed disjoint from the previous bad pairs. 

Finding a sequence avoiding one such a bad pair is easy. However, if there are many bad pairs given, finding one new sequence avoiding all those is non-trivial. 
However, such a problem can be reduced to a graph theory problem regarding graph expansion. Indeed, there is a natural correspondence between such a sequence and a path in a certain `grid-like' directed graph avoiding given vertices.
Exploiting this correspondence and solving this graph problem, we will prove Lemma~\ref{lem: many intervals} which is a key technical lemma for the proof of Theorem~\ref{thm: counting}. We give a slightly more detailed sketch on these ideas right before the statement of Lemma~\ref{lem: many intervals}.

\subsection{Approximating a number with integral combinations of $m_{s}(k)$}

In order to find a desired integral combination for a given number $Q$, it is convenient to define the following infinite vectors and their inner product:
\begin{align*}
\mathbf{m}^k&:= ( m_1(k), m_2(k), \dots), \\
\mathbf{id}^s&:= (1^s, 2^s, \dots )
\enspace \text{and}\\  \mathbf{b}^{t,s} &:= ( \underbrace{0,\dots,0}_{s-1 \text{ zeros}} , -\binom{t}{1}, 
\underbrace{0,\dots,0}_{s-1 \text{ zeros}}, (-1)^2 \binom{t}{2},
\dots \dots , (-1)^t\binom{t}{t},0,0,\dots).
\end{align*}
%In the above notation, we sometimes omit $k$ and write $\mathbf{m}$ or $m_s$ instead of $\mathbf{m}^k$ or $m_{s}(k)$ if it is clear from the context.

If we consider a finite vector $(x_i:i\in S)$ for $S\subseteq \mathbb{N}$, we naturally identify it with an  infinite vector $(x_i: i\in \mathbb{N})$ by letting $x_i=0$ if $i\notin S$. 
This allows us to define inner product for two vectors with different lengths. In other words, for two vectors $\mathbf{x}$ and $\mathbf{y}$, we define 
$$\mathbf{x}\cdot \mathbf{y} = \sum_{i\in \mathbb{N}} x_iy_i.$$
%Sometimes, we will consider a vector as a function defined on a subset of $\mathbb{N}$.
%If $f$ is a function defined on a subset $S\subseteq \mathbb{N}$, then it is identified with a vector $(f_1,f_2,\dots)$ where $f_x=0$ when $x\notin S$. This yields a natural inner product of two functions as $f\cdot g = \sum_{i\in \mathbb{N} } f(i)g(i)$ if the sum converge.
For a given vector $\mathbf{x}$ and a set $S\subseteq \mathbb{N}$, we let $\mathbf{x}\restriction_{S} := (x_i \mathbf{1}_S(i) )_{i\in \mathbb{N}}$ be the vector we obtain where $\mathbf{1}_S(i)=1$ if $i\in S$ and $0$ if $i\notin S$.

In fact, later we will choose a number $a$ from an interval $P=[x,x+c]$ and a number $b$ from an interval $N=[x',x'+c]$ and express $b-a$ as an integral combination of the numbers $m_{s}(k)$.
In this scenario, if any number between $x'-x+c$ and $x'-x-c$ can be expressed as a non-negative integral combination of the $m_{s}(k)$'s, then we obtain a monochromatic solution.
Hence, for some number $Q= x'-x$, finding a non-negative integral combinations of $m_{s}(k)$ whose value is roughly $Q$ is sufficient.
In other words, we wish to find a non-negative vector $\mathbf{v}$ such that $\mathbf{v}\cdot \mathbf{m}^k = Q\pm c$.

 As this number $Q$ is obtained from the intervals of our choice, we have some control on the value of $Q$. In particular, we can take $Q$ to be very large.
With such a large number $Q$, we can reduce our problem to finding an integral vector $\mathbf{v}'$, which need not be non-negative. 

Indeed, if $Q$ is sufficiently large, then we can find a large number $q$ so that $Q'=Q- q \sum_{s\in [k_0-k]} m_s(k)$ is still positive. 
Then if we can find another vector $\mathbf{v}' \in [-q,q]^{k_0-k}$ such that $\mathbf{v}'\cdot \mathbf{m}^k = Q'\pm c$, then the non-negative vector $(q,\dots, q) + \mathbf{v}'$ provides a desired non-negative integral combination. 
The following lemma allows us to express a given number $Q'$ as an integral combination whose coefficients are allowed to be negative while 
having small absolute values, hence providing such a vector $\mathbf{v}'$. In addition, we can restrict the support of such a vector and this fact will be also important for our application.

\begin{lemma}\label{lem: estimation}
Suppose $0< \frac{1}{k} \ll \frac{1}{\ell} \ll \frac{1}{q} \ll \frac{1}{\|p\|}, \frac{1}{d} \leq \frac{1}{2}$. For a given positive integer $Q'$ at most $q k^{d-1}$, there exists an integer vector $\mathbf{v}'\in [-\ell,\ell]^{k_0-k}$ satisfying
$$\mathbf{v}'\cdot \mathbf{m}^k = Q' \pm \ell$$
and $\mathbf{v}'_i=0$ for all $i \notin \{ t 2^{j} : t\in [d-1], 1\leq j \leq q+ \frac{1}{2} \log{k}\}$.
\end{lemma}

In order to prove this lemma, we collect several useful propositions.
First, note that expressing a number $Q$ as a non-negative integral combination of a geometric series $1,t,t^2,\dots$ is easy. We simply take the largest $i$ and $c\in [t-1]$ such that $c t^i\leq Q$ and reduce the problem to the problem of expressing $Q- c t^{i}$ as a non-negative integral combination of $1,t,t^2,\dots$. In fact, this greedy approach works with other series as long as the series does not grow too fast. The following proposition confirms this. As the proof is just what we described, we omit the proof.

\begin{proposition}\label{prop: binary expansion}
Let $t>1$ be an integer and $c_1,\dots, c_m$ be a sequence of positive numbers satisfying $c_{i}< c_{i+1} <  t c_{i}$ for all $i\in [m-1]$. 
If $Q$ is a positive number at most $tc_m$, then there exist numbers $b_1,\dots, b_m \in [0,t-1] $ satisfying $\sum_{i=1}^{m} b_i c_i = Q\pm c_1$.
\end{proposition}

As we wish to express a number $Q$ as an integral combination of the numbers $m_{s}(k)$, it is useful to know the properties of the numbers $m_{s}(k)$. In fact, if we compute $m_{s}(k) = p(k+s)-p(k)$ term by term, we obtain the following $k$-ary expansion.
\begin{align}\label{eq: mik}
m_s(k) = \sum_{i=0}^{d-1} \left( \sum_{r=i+1}^{d} a_{r} \binom{r}{i} s^{r-i}  \right)k^{i} ~\text{ and }~ 
\mathbf{m}^k = \sum_{i=0}^{d-1} \left( \sum_{r=i+1}^{d} a_{r} \binom{r}{i} \mathbf{id}^{r-i}  \right)k^{i}
\end{align}
Based on this, if $1/k  \ll 1/\ell \ll 1/d,1/\|p\|$ and $s\leq  2^{2q}k^{1/2}$, we have 
\begin{align}\label{eq: mk property}
m_s(k) = a_d  s d  k^{d-1} \pm \ell k^{d-3/2}.
\end{align}
In order to prove Lemma~\ref{lem: estimation}, we plan to use the vectors $\mathbf{b}^{t,s}$.
The following proposition allows us to understand the $k$-ary expansion of $\mathbf{m}^k \cdot \mathbf{b}^{t,s}$.

\begin{proposition}\label{prop: mb size}
Suppose $0< \frac{1}{q} \ll \frac{1}{\|p\|}, \frac{1}{d}<1$ and $t\in [d]$. Then there exist integers $h_{t,0},h_{t+1,0},$ $h_{t+1,1},h_{t+2,0},\dots, h_{d,d-t} \in [-q,q]$ satisfying the following for all $s,k\in \mathbb{N}$.
\begin{align}\label{eq: binomial equality}
\mathbf{m}^{k}\cdot \mathbf{b}^{t,s}
=  \sum_{i=0}^{d-t} \sum_{r=t+i}^{d} h_{r,i} s^{r-i} k^i.
\end{align}
Moreover, $h_{d,d-t}$ is always non-zero. If $t\leq d-1$, then $h_{d,d-t-1}$ is also non-zero.
\end{proposition}
\begin{proof}
We first collect the following three standard facts.
As these are elementary facts, we state them without proofs.

First, it is well-known that $\mathbf{id}^1\restriction_{[t]},\dots, \mathbf{id}^{t-1}\restriction_{[t]}$ are linearly independent vectors in $\mathbb{R}^{t}$.\COMMENT{
If $\sum_{i\leq t-1} c_i \mathbf{id}^i = \mathbf{0}$, then the polynomial $c_1x+ c_2x^2+\dots + c_{t-1}x^{t-1}$ has roots $0,1,2,\dots, t$, a contradiction as a polynomial of degree $t-1$ has at most $t-1$ solutions.}
Second, the following holds for each $i\in [t-1]$.
\begin{align}\label{eq: dot zero b}
 \mathbf{id}^i \cdot \mathbf{b}^{t,s} = 0.
 \end{align}
\COMMENT{
Consider $(1-z)^t$ and take derivative for $0,1,2,\dots, t-1$ times. Put $z=1$ to obtain equalities.
Take an appropriate linear combinations of these equalities, then we can conclude this.
}
Third, for each $i\geq t$, the vector $\mathbf{id}^{i}\restriction_{[t]}$ is not a linear combination of 
$\mathbf{id}^1\restriction_{[t]},\dots, \mathbf{id}^{t-1}\restriction_{[t]}$. 
\COMMENT{
If $\mathbf{id}^i + \sum_{j\leq t-1} c_j \mathbf{id}^j = \mathbf{0}$, then the polynomial $x^i + c_{t-1}x^{t-1} + \dots + c_2 x^2 + c_1x$ has roots $0,1,2,\dots, t$. However, it is known that such a polynomial has at most $t-1$ nonzero distinct roots (which can be proved by using induction). Thus we have what we want.
}

The first fact and \eqref{eq: dot zero b} with $s=1$ together imply that 
 ${\rm Span}(\mathbf{id}^1\restriction_{[t]},\dots, \mathbf{id}^{t-1}\restriction_{[t]})$ is the orthogonal complement of the vector $\mathbf{b}^{t,1}$ in $\mathbb{R}^t$.
Thus the third fact implies that, for each $i \in [t,d]$, the vector $\mathbf{id}^{i}\restriction_{[t]}$ does not lie in the orthogonal complement of $\mathbf{b}^{t,1}$ in $\mathbb{R}^{t}$. In other words, we have
\begin{align}\label{eq: id dot b}
\mathbf{id}^{i} \cdot \mathbf{b}^{t,s} = (\mathbf{id}^{i}  \cdot \mathbf{b}^{t,1})s^{i} = \left\{\begin{array}{ll}
0 & \text{if } i\in [t-1] \\
\text{$c'(t,i) s^{i}$}  & \text{if } i\in [t,d]
\end{array}\right.
\end{align}\COMMENT{We have an upper bound on the absolute value of $\mathbf{id}^{i}\restriction_{[t]}  \cdot \mathbf{b}^{t,1} $ which is purely in terms of $d$ and $t$ and $1/q\ll 1/d \leq 1/t$.}
for some non-zero integer $c'(t,i)$ with $|c'(t,i)|\leq q^{1/2}$ as $1/q \ll 1/d.$
By using \eqref{eq: mik}, the following equality is straightforward for each $t\in [d]$.
$$
\mathbf{m}^{k} \cdot \mathbf{b}^{t,s} = \sum_{i=0}^{d-1} \left( 
 \sum_{r=i+1}^{d} a_r \binom{r}{i} (\mathbf{id}^{r-i} \cdot \mathbf{b}^{t,s})
\right) k^i.
$$
However, \eqref{eq: id dot b} implies that $ \mathbf{id}^{r-i} \cdot \mathbf{b}^{t,s}$ is zero for $r-i<t$ and is $c'(t,r-i)s^{r-i}$ otherwise. So we obtain
$$ \mathbf{m}^k \cdot \mathbf{b}^{t,s} 
= \sum_{i=0}^{d-t} \left(\sum_{r=t+i }^{d} a_r \binom{r}{i} c'(t,r-i) s^{r-i}\right) k^{i}.$$
By letting 
$$h_{r,i}= a_r \binom{r}{i} c'(t,r-i)$$ for each $r\in [t,d]$ and $i\in [0,d-t]$, we obtain the desired inequality $|h_{r,i}|\leq q$ as $|c'(t,r-i)|\leq q^{1/2}$ and $1/q \ll 1/\|p\|$. 

 Furthermore, because \eqref{eq: id dot b} implies that $c'(t,i)$ is nonzero for $i\geq t$, the number $h_{d,d-t} = a_d\binom{d}{d-t} c'(t,t)$ is also nonzero. If $t\leq d-1$, then \eqref{eq: id dot b} implies that $c'(t,t+1)$ is non-zero, so the number $h_{d,d-t-1} = a_d\binom{d}{d-t-1} c'(t,t+1)$ is also a nonzero number, confirming the moreover part of the statement. This proves the proposition.
\end{proof}

The previous proposition provides a $k$-ary expression of a number $\mathbf{m}^k \cdot \mathbf{b}^{t,s}$ with small coefficients. 
To prove Lemma~\ref{lem: estimation}, we will find our desired vector $\mathbf{v}'$ as a sum of several $\mathbf{b}^{t,s}$, then Proposition~\ref{prop: mb size} provides a better prediction on the value of $\mathbf{m}^k\cdot \mathbf{v}'$. 

\begin{proof}[Proof of Lemma~\ref{lem: estimation}]
We plan to define several vectors
$\mathbf{c}^{d-1,1}, \mathbf{c}^{d-1,2},\dots, \mathbf{c}^{d-1,w_{d-1}}, \mathbf{c}^{d-2,1},\dots \mathbf{c}^{1,w_1}$ in such a way that $c_{t,i}:= \mathbf{c}^{t,i}\cdot \mathbf{m}^k$ forms a sequence satisfying the conditions of Proposition~\ref{prop: binary expansion}. This would give us appropriate numbers $b_{d-1,1},\dots, b_{1,w_{1}}$ in such a way that 
$\sum b_{t,i} c_{t,i} = Q\pm \ell$ holds, then the vector $\mathbf{v}'= \sum b_{t,i} \mathbf{c}^{t,i}$ will be our desired vector with $\mathbf{v}'\cdot \mathbf{m}^k = \sum_{i} b_{t,i}c_{t,i} = Q\pm \ell$.

For this purpose, 
for each $t\in [d-1]$ and $i\in \mathbb{N}$,
we choose $\mathbf{c}^{t,i}$ to be either  $$\mathbf{b}^{t,2^{q(i+1)} } -  2^{qt} \mathbf{b}^{t, 2^{qi} } \text{ or } -(\mathbf{b}^{t,2^{q(i+1)} } -  2^{qt} \mathbf{b}^{t, 2^{qi} })$$
in such a way that $c_{t,i}:= \mathbf{c}^{t,i}\cdot \mathbf{m}^{k}$ is non-negative. Note that $\mathbf{c}^{t,i}$ has nonzero values only on its $j$-th coordinate for $j \in \{ t 2^{j'} : t\in [d-1], j'\in \{qi, q(i+1)\} \}$.

We now verify that these choices ensure that the sequence $c_{t,i}$ satisfies the conditions of the Proposition~\ref{prop: binary expansion}. 
 By Proposition~\ref{prop: mb size}, we have the following equality with some integers $h_{r,j}$ satisfying $|h_{r,j}|\leq q$.
\begin{eqnarray}\label{eq: c expression}
c_{t,i} &=& \left|\sum_{j=0}^{d-t} \sum_{r=t+j}^d h_{r,j}  2^{q(i+1)(r-j)}k^{j} - 2^{qt} \cdot \sum_{j=0}^{d-t} \sum_{r=t+j}^d h_{r,j} 2^{qi(r-j)} k^{j}\right| \nonumber \\
&= & \left|\sum_{j=0}^{d-t-1} \sum_{r=t+j+1}^{d} h_{r,j} (2^{q(r-j)} - 2^{qt}) 2^{qi(r-j)}k^{j}\right|.
\end{eqnarray}
The final equality holds as the term for $r=t+j$ vanishes because $2^{q(r-j)}-2^{qt}$ is zero when $r=t+j$.
Moreover, if $k$ is much bigger than $q$, then the term with highest exponent in $k$ becomes much bigger than all other terms. The following claim provides this fact.
\begin{claim}\label{cl: increasing}
    If $i\in \mathbb{N}$ satisfies $2^{q(i-1)}\leq k^{1/2}$, then we have 
\begin{align}\label{eq: increasing}
c_{t,i-1} < c_{t,i} \leq 2^{2qd} c_{t,i-1} 
\enspace \text{ and } \enspace 
q 2^{qi(t+1)} k^{d-t-1} < c_{t,i} < q 2^{q(i+1)(t+1)} k^{d-t-1}.
\end{align}
\end{claim}
\begin{proof}
Note that $1/k \ll 1/q \ll 1/\|p\|$ implies $d^2 q 2^{2qd} < k^{1/8}$. As $|h_{r,j}|\leq q$, we have 
\begin{eqnarray*}
\left| \sum_{j=0}^{d-t-2} \sum_{r=t+j+1}^{d} h_{r,j} (2^{q(r-j)} - 2^{qt}) 2^{qi(r-j)}k^{j}  \right|
&\leq& \sum_{j=0}^{d-t-2} \sum_{r=t+j+1}^{d} q 2^{qd} 2^{qi(t+1+r-t-j-1)} k^j\\
&\leq& 
\sum_{j=0}^{d-t-2} \sum_{r=t+j+1}^{d} 
q 2^{2qd} 2^{qi(t+1)} 2^{q(i-1)(r-t-j-1)} k^{j}  \\
&\leq& 
\sum_{j=0}^{d-t-2} \sum_{r=t+j+1}^{d} 
d^{-2} 2^{qi(t+1)}  k^{(r-t-j-1)/2+j+1/8}  \\
&\leq & 2^{qi(t+1)} k^{d-t - 5/4}.
\end{eqnarray*}
Here, the second inequality holds as $r-t-j-1\leq d$ and
the final inequality holds as $(r-t-j-1)/2+j \leq d-t-3/2$ for $j\leq d-t-2$ and $r\leq d$.
As $t\in [d-1]$, the number $h_{d,d-t-1}$ is non-zero by Proposition~\ref{prop: mb size}, so we obtain
\begin{eqnarray*}
c_{t,i}&=&  \left|h_{d,d-t-1} (2^{q(t+1)} - 2^{qt})2^{qi(t+1)} (k^{d-t-1} \pm k^{d-t-5/4} )\right|.
\end{eqnarray*}
Moreover, $h_{d,d-t-1}$ is a non-zero integer in $[-q,q]$ and $1/k\ll 1/q,1/\|p\|$ and $q<2^{q(t+1)}-2^{qt} \leq 2^{q(t+1)}$, the above equality yields the second desired inequality as follows:
\begin{align*}%\label{eq: cii-1}
    q 2^{qi(t+1)} k^{d-t-1} < c_{t,i} < q 2^{q(i+1)(t+1)} k^{d-t-1}.
\end{align*}
As this inequality holds for $i-1$ as well,
this implies that
$$c_{t,i-1}<q 2^{qi(t+1)} k^{d-t-1} < c_{t,i} <q 2^{q(i+1)(t+1)} k^{d-t-1} < 2^{2q(t+1)} c_{t,i-1} \leq 2^{2qd} c_{t,i-1}$$
and proves the claim.
\end{proof}

Let $w_1$ be the number satisfying $2^{q(w_1-1)} \leq  k^{1/2} < 2^{q w_1}$.
Then Claim~\ref{cl: increasing} with this choice of $w_1$ implies that 
$c_{1,1},\dots, c_{1,w_1}$ is a sequence satisfying \eqref{eq: increasing}. Moreover, we obtain 
$$ c_{1,w_1} > q 2^{2qw_1} k^{d-2} \geq q k^{d-1}.$$

For $t\geq 2$, assume that we have defined $w_{t-1}$.
Consider the number $w'$ satisfying $2^{q(w'-1)}\leq k^{1/2} \leq 2^{qw'}$.
Then Claim~\ref{cl: increasing} with this choice implies that
$$c_{t,w'} > q 2^{qw'(t+1)} k^{d-t-1}
\geq q k^{(t+1)/2} k^{d-t-1} \geq q k^{d-t} > 2^{-2qd} c_{t-1,1}.
$$
Here, the final inequality holds by \eqref{eq: increasing}.
Now let $w_t$ be the smallest number such that $c_{t,w_{t}} \leq c_{t-1,1} \leq 2^{qd} c_{t,w_t}$. 
By the definition of $w_t$ and the above inequality, $w_t$ exists and it has value at most $w'$. Hence, Claim~\ref{cl: increasing} applies for all $i\leq w_{t}$, thus the sequence $c_{t,1},\dots, c_{t,w_{t}}$ satisfies \eqref{eq: increasing}. Repeating this yields the sequence 
$$c_{d-1,1},\dots,c_{d-1,w_{d-1}}, c_{d-2,1},\dots,c_{1,w_1}$$
satisfying
that a term in the sequence is bigger than the previous term and at most $2^{2qd}$ times the previous term.
Furthermore, $c_{d-1,1}$ can be written as a function of $q,d$ and the coefficients of the polynomial $p(z)$,
the hierarchy $1/\ell \ll 1/q, 1/\|p\|$ implies $c_{d-1,1} \leq \ell$. 

Proposition~\ref{prop: binary expansion} implies that there exists numbers $b_{t,i} \in [2^{2qd}]$ satisfying 
$$\sum_{t\in [d-1]}\sum_{i\in [w_t]} b_{t,i} c_{t,i}  = Q\pm \ell.$$
Let $\mathbf{v}' = \sum_{i\in [w_t]} b_{t,i} \mathbf{c}_{t,i}$. Then it is easy to see that each coordinate of $\mathbf{c}_{t,i}$ has absolute value at most $\ell$ as $1/\ell \ll 1/q , 1/\|p\|$. 
Moreover, by our choices of $\mathbf{c}_{t,i}$ it is easy to see that $\mathbf{v}_i=0$ for all $i \notin \{ t 2^{j} : t\in [d-1], 1\leq j \leq q+ \frac{1}{2} \log{k}\}$. This proves the lemma.
\end{proof}

The following proposition establishes an `independence' over the numbers $m_{s}(k)$, by showing that any nontrivial integral combination of $t<d$ numbers $m_{1}(k), m_{2}(k), \dots, m_{t}(k)$ is far from zero.

\begin{proposition}\label{prop: t terms independent}
Suppose $0<1/k \ll 1/q \ll 1/\|p\|, 1/d \leq 1/2$ and $t\in [d-1]$.
For any nonzero vector $(c_1,\dots, c_{d-1})\in [- \frac{k}{q},\frac{k}{q}]^{t}$, we have
$|\sum_{i\in [t]} c_{i} \cdot m_{i}(k))| >\frac{1}{2} k^{d-t}$.
\end{proposition}
\begin{proof}
For each $i\in \mathbb{N}$, let $\mathbf{v}^i\in [-q,q]^d$ be an integer vector such that for each $0\leq j\leq d-1$, the $(j+1)$-th coordinate of the vector $\mathbf{v}^i$ is
$$(\mathbf{v}^i)_{j+1} = \sum_{r=j+1}^{d} a_r \binom{r}{j} i^{r-j}.$$
Note the resemblance of this vector with the `$k$-ary expansion' of $m_s(k)$ in the expression \eqref{eq: mik}. 

We will first show that at least one of the last $t$ coordinates of $\sum_{i\in [t]} c_i \mathbf{v}^{i }$ has a non-zero absolute value. Let $U$ be a $d\times t$ matrix whose columns are $\mathbf{v}^{1}, \mathbf{v}^{2},\dots, \mathbf{v}^{t}$. 
Let $U_1,\dots, U_{d}$ be the rows of $U$ and let $U'$ be the submatrix of $U$ with rows $U_{d-t+1},\dots, U_{d}$.
Then for each $0\leq j\leq d-1$, we have
\begin{align*}
U_{j+1}&=\sum_{r=j+1}^{d} a_{r}\binom{r}{j}\left(1^{r-j},2^{r-j},\dots, t^{r-j} \right )  
= a_d \binom{d}{j}  \mathbf{id}^{d-j}_{[t]} + 
\sum_{r=j+1}^{d-1} a_{r}\binom{r}{j}\mathbf{id}^{r-j}_{[t]}.
\end{align*}
Hence $U_{j+1}$ is a linear combination of $\mathbf{id}^{d-j}_{[t]}, \dots, \mathbf{id}^{1}_{[t]}$, with the coefficient of $\mathbf{id}^{d-j}_{[t]}$ being nonzero. It is a standard fact that $\{\mathbf{id}^{t}_{[t]}, \dots, \mathbf{id}^{1}_{[t]}\}$ forms a basis of $\mathbb{R}^{t}$. Using this basis,
the above inequality implies that $U'$ is row-equivalent to an upper triangular matrix with nonzero diagonals, hence $U'$ is a full rank $t\times t$ matrix.
Thus, for any nonzero $(c_1,\dots, c_t)$, at least one of the last $t$ coordinates of $\sum_{i=1}^{t} c_{i} \mathbf{v}^{i}$ is nonzero. 
 Assume that $r\in [d-t,d-1]$ is the largest number such that the $(r+1)$-th coordinate of $\sum_{i=1}^{t} c_{i} \mathbf{v}^{i}$ is nonzero.
As we have $m_{i}(k) = \sum_{j=0}^{d-1} (\mathbf{v}^{i})_{j+1} k^j$, we conclude that
$$\Big|\sum_{i=1}^{t} c_{i} \cdot m_{i}(k))\Big| = \Big|\sum_{j=0}^{d-1}\sum_{i=1}^{t} c_{i} (\mathbf{v}^{i})_{j+1} k^{j} \Big| 
\geq  k^{r} - \sum_{i=0}^{r-1} q  k^{i} > \frac{1}{2}  k^{d-t}.$$
Here, the penultimate inequality holds as $1/q \ll 1/\|p\|$.
\end{proof}

\subsection{Other tools}
The following proposition is also useful for identifying the structure of a coloring with no bad pairs.
\begin{proposition}\label{prop: sum}
    Let $n,m,Q$ be integers with $Q\geq  2n+4m$ where $m$ is odd.
    Let $\phi: [n,Q-n]\rightarrow \{-1,1\}$ be a coloring satisfying the following for any $a,b\in [n,Q-n]$.
    \begin{enumerate}[label={\rm (A\arabic*)}]
        \item \label{A1} $\phi(a)\leq \phi(a+m)$ for all $a \in [n,Q-n-m]$.
        \item \label{A2} If $a+b = Q$, then at least one of $a,b$ has color $-1$.
        \item \label{A3} If $a+b = Q+m$, then at least one of $a,b$ has color $+1$.
    \end{enumerate}
    Then any number
    in $[Q+m+1,2Q-2n-2m]$ can be written as a sum of two numbers of color $+1$ and any number in $[2n+2m,Q-1]$ can be written as a sum of two numbers of color $-1$.
\end{proposition}
\begin{proof}
For each $j \in \mathbb{N}$, consider the arithmetic sequence $\mathbf{a}^j$ of difference $m$ containing all numbers in $[n,Q-n]$ which are congruent to $j$ modulo $m$. Then it is easy to see that 
$\min(\mathbf{a}^{j}) \leq n+m$ and $\max(\mathbf{a}^{j}) \geq Q-n-m$.

Assume $Q\equiv j_*$ (mod $m$) for some $j_*\in \mathbb{Z}_m$. Let $A$ be the set of those $j\in \mathbb{Z}_{m}$ for which the sequence $\mathbf{a}^j$ does not contain any number $x$ of color $-1$ satisfying $x> Q/2$.
Similarly, let $B$ be the set of those $j\in \mathbb{Z}_m$ for which the sequence $\mathbf{a}^j$ does not contain any number $x$ of color $+1$ satisfying $x< Q/2+m/2$.

We now claim that $|A|> m/2$.
To show this, suppose that neither $j$ nor $j_*-j$ is in $A$ for some $j\in \mathbb{Z}_m$.
Then $\mathbf{a}^j$ contains a number $b > Q/2$ of color $-1$ and $\mathbf{a}^{j_*-j}$ contains a number $b'> Q/2$ of color $-1$.
Then we have $b+b'> Q$ and $b+b'\equiv j_*\equiv Q$ (mod $m$), so $b+b'\geq Q+m$.
On the other hand, by \ref{A1}, any number in $\mathbf{a}^j$ smaller than or equal to $b$ has color $-1$ and any number in $\mathbf{a}^{j_*-j}$ smaller than or equal to $b'$ has color $-1$. Hence the sums $b+b', (b-m)+b', b+ (b'-m),\dots, \min(\mathbf{a}^j) + \min(\mathbf{a}^{j_*-j})$ yield all numbers in $b+b', b+b'-m,\dots ,2n+2m+j_*$ and this contains $Q+m$ because $2n+2m+j_* \leq Q +m \leq b+b'$. This  contradicts \ref{A3}.
Hence, at least one of $j$ and $j_*-j$ must belong to $A$. As $m$ is an odd number, this implies $|A|>m/2$.

Furthermore, we also claim that $|B|>m/2$.
Similarly as before, suppose that neither $j$ nor $j_*-j$ is in $B$ for some $j\in \mathbb{Z}_m$.
Then $\mathbf{a}^j$ contains a number $b < Q/2+m/2$ of color $+1$ and $\mathbf{a}^{j_*-j}$ contains a number $b'< Q/2+m/2$ of color $+1$.
Thus $b+b'< Q+m$ and $b+b'\equiv j_*\equiv Q$ (mod $m$), thus $b+b'\leq Q$.
On the other hand, by \ref{A1}, any number in $\mathbf{a}^j$ bigger than or equal to $b$ has color $+1$ and any number in $\mathbf{a}^{j_*-j}$ bigger than or equal to $b'$ has color $+1$. Hence, the sums $b+b', (b+m)+b', b+ (b'+m),\dots, \max(\mathbf{a}^j) + \max(\mathbf{a}^{j_*-j})$ yield all numbers in $b+b', b+b'+m, \dots, 2Q-3m+j_*$ and this contains $Q$ because $b+b'\leq Q \leq 2Q-2n-3m$. This contradicts \ref{A2}.
Hence at least one of $j$ and $j_*-j$ must belong to $B$. As $m$ is an odd number, this implies $|B|>m/2$.

As both $A$ and $B$ contains more than $m/2$ numbers in $\mathbb{Z}_m$, it is easy to see the following
$$ A+A = \mathbb{Z}_m = B+B.$$
Hence, for any number $x\in [Q+m+1, 2Q-2n-2m]$, there exist $j,j'\in A$ with $x \equiv j+j'$ (mod $m$).
Let $b$ be the smallest number of color $+1$ in $\mathbf{a}^j$ and $b'$ be the smallest number of color $+1$ in $\mathbf{a}^{j'}$.
As $j,j'\in A$, those numbers $b$ and $b'$ exist and satisfy $b\leq Q/2 +m$ and $b'\leq Q/2+m$. 
Thus $b+b'\leq Q+2m$, and $b+b'\equiv x$ (mod $m$).
This implies that the sums $b+b', (b+m)+b', b+ (b'+m),\dots, \max(\mathbf{a}^j) + \max(\mathbf{a}^{j'})$ yield all numbers in $[Q+m+1,2Q-2n-2m]$ which are congruent to $x$ modulo $m$. Here, we used the fact that $ \max(\mathbf{a}^j) + \max(\mathbf{a}^{j'}) \geq 2Q-2n-2m$.
As this contains $x$, \ref{A1} implies that the number $x$ can be written as a sum of two numbers of color $+1$.

Similarly, for any number $y\in [2n+2m, Q-1]$, there exist $j,j'\in B$ with $y\equiv j+j'$ (mod $m$).
Let $b$ be the largest number of color $-1$ in $\mathbf{a}^j$ and $b'$ be the largest number of color $-1$ in $\mathbf{a}^{j'}$.
As $j,j'\in B$, those numbers $b$ and $b'$ exist and satisfy $b\geq Q/2-m/2$ and $b'\geq Q/2 - m/2$. Thus $b+b' \geq Q-m$ and $b+b'\equiv y$ (mod $m$).
This implies that  the sums $b+b', (b-m)+b', b+ (b'-m),\dots, \min(\mathbf{a}^j) + \min(\mathbf{a}^{j'})$ yield all numbers in $[2n+2m,Q-1]$ which are congruent to $y$ modulo $m$. Here, we used the fact that $ \min(\mathbf{a}^j) + \min(\mathbf{a}^{j'}) \leq 2n+2m$.
Thus this contains $y$, \ref{A1} implies that the number $y$ can be written as a sum of two numbers of color $-1$.
\end{proof}

Consider the following \emph{grid digraph} $G(n)$ with the vertex set $[0,n]\times [0,n]$ and the arc set consisting of the directed edges from $(i,j)$ to $(i,j+1)$ and from $(i,j)$ to $(i+1,j)$. Furthermore,  let $G(n;m)$ be the induced subgraph of $G(n)$ induced by the set of all the vertices $(i,j) \in [0,n]\times [0,n]$ with $|i-j|\leq m$. 
We call $G(n;m)$ a \emph{reduced grid digraph} and we write $L_s$ to denote the set of all vertices $(i,j)$ in $G(n;m)$ with $i+j=s$, and call this set \emph{level $s$} of the graph $G(n;m)$. Note that $|L_s|\leq m+1$ in $G(n;m)$ for any $s\in [0,2n]$.
We write $L(s,t)= \bigcup_{i\in [s,t]} L_i$ for the set of all vertices of level $j$ for all $s\leq j\leq t$. Then the following lemma says that if a vertex set $Z$ is small and `thin', then we can find a path from $(0,0)$ to $(n,n)$ avoiding the vertices of $Z$.

\begin{lemma}\label{lem: grid}
    Suppose $0<1/n\ll \eta \ll 1/\ell \ll 1$ and 
    let $m= \eta^2 n$ be an integer.
    Let $Z\subseteq G(n;m)$ be a set of vertices satisfying the following.
\begin{enumerate}[label={\rm (G\arabic*)}]
\item \label{G1}\label{G1} $|Z|\leq \frac{1}{10}m$.
\item \label{G2}\label{G2} For each $i\leq \frac{m}{\ell}$, both sets $L(0,i\ell)$ and $L(2n-i\ell, 2n)$ contains at most $(i-1)$ vertices of $Z$.
\end{enumerate}
    Then $G(n;m)$ contains a directed path $P$ from the vertex $(0,0)$ to the vertex $(n,n)$ which does not intersect $Z$.
\end{lemma}
\begin{proof}
We first prove the following claim.
\begin{claim}\label{cl: expands}
For $i\in [2n]$ and a nonempty set $A\subseteq L_i$, its out-neighborhood $N^+(A)$ in the graph $G(n;m)$ satisfies the following.
\begin{enumerate}
    \item[(a)] If $i< m$, then $|N^+(A)|\geq |A|+1$.
    \item[(b)] If $m\leq i\leq n-1$ and $i$ has different parity from $m$, then $|N^+(A)|\geq |A|+1$.
    \item[(c)] If $m\leq i\leq n-1$ and $i$ has the same parity as $m$, then $|N^+(A)|\geq |A|-1$.
\end{enumerate}
\end{claim}
\begin{proof}
    If $i<m$ or $i$ has different parity from $m$, then we consider 
    $$A_0 = \{ (i+1,j): (i,j)\in A\} \enspace \text{and}\enspace A_1 = \{(i,j+1):(i,j)\in A\}.$$
    For such a number $i$, it is easy to see that these two sets have the same size as $A$ and they are not identical, we have $|N^+(A)|= |A_0\cup A_1| \geq |A|+1$, yielding (a) and (b).

    If $m\leq i\leq n-1$ and $i$ has the same parity as $m$, then $A_0$ has size at least $|A|-1$, so (c) holds.
\end{proof}
    For each $i\in [2n]$, let $Z_i = Z\cap L_i$ be the set of vertices in $Z$ which belong to the $i$-th level. 
Consider the graph $G'= G(n;m)- Z$. 
    For each $i\in [m]$, let $S_i$ be the set of vertices $v$ at level $i$ such that $G'$ contains a directed path from $(0,0)$ to $v$.
Note that for $S\subseteq (L_i\setminus Z)$, we have 
$$N^+_{G'}(S) = N^+_{G(n;m)}(S)\setminus Z_{i+1}.$$
Note that \ref{G2} with $i=1$ implies that $L(0,\ell)$ does not contain any vertices of $Z$. Thus we have $|S_\ell|= \ell+1$. For $t\in [m-\ell]$, we can apply Claim~\ref{cl: expands}~(a) $t$ times to obtain the sizes of the $t$-th out-neighborhood of $S_\ell$ in the digraph $G'= G(n;m)- B$ to
conclude 
$$|S_{\ell+t}| \geq \ell+t - \sum_{i\in [t] }|Z_{\ell+i}| \geq \ell+t -\frac{t}{\ell}
\geq \ell + (1-\frac{1}{\ell})t.$$
Indeed, this shows that the set $S_{\ell+t}$ never become an empty set for $t\in [m-\ell]$, thus we can again apply Claim~\ref{cl: expands} to obtain the inequality for $t+1$ and so on.
From this, we conclude that $|S_{m}|\geq \ell + (1-\frac{1}{\ell})m \geq \frac{4}{5} m$.

Again, we apply Claim~\ref{cl: expands}~(b) and (c) repeatedly to the vertex set $S_{m}$ to obtain the size of the $t$-th out-neighborhood $S_{m+t}$ of $S_{m}$ in $G'$ for $t\in [n-m]$. As $|Z|\leq \frac{1}{10} m$, we have
$$|S_{n}| \geq \frac{4}{5}m -1 - |Z| \geq \frac{2}{3}m.$$

Similarly, we consider paths towards the vertex $(n,n)$. A symmetric argument again gives that there are at least $\frac{2}{3}m$ vertices $v$ in $S_n$ which admit a path from $v$ to $(n,n)$ in $G'$. As $|L_n|\leq m+1 < \frac{2}{3}m+ \frac{2}{3}m$, there exists a vertex $v\in L_n$ such that $G'$ admits a path from $(0,0)$ to $v$ and from $v$ to $(n,n)$. These two paths together yield a desired path.
\end{proof}

Finally, we need the following inequality.

\begin{lemma}\label{lem: inequality}
For $m,n\in \mathbb{N}$,
    let $x_1,\dots, x_n$ be a sequence satisfying $\sum_{j\in [i]} j x_j \leq (i-1)m$ for all $i\in [n]$.
    Then we have $\sum_{j\in [n]} x_j \leq 2m(1+\log n)$.
\end{lemma}
\begin{proof}
    Suppose $\sum_{j\in [n]} x_j > 2m(1+\log n)$.
    For each  $0\leq i\leq  \log n$ consider the intervals $I_i = [2^{i},2^{i+1}]$.
    By averaging, there exists $0 \leq i\leq \log n$ satisfying $\sum_{j\in I_i}x_j > 2m.$
    However, this implies
    $$\sum_{j\leq 2^{i+1}} j x_j 
    \geq \sum_{j\in I_i} 2^{i} x_j > 2^{i+1} m,$$
    which contradicts the assumption of the lemma. This proves the lemma.
\end{proof}

\section{Counting monochromatic solutions}

Throughout this section, we assume that a non-odd polynomial $p(z)$ of degree $d\geq 2$ with a positive leading coefficient and
a $2$-coloring $\phi: \mathbb{N}\rightarrow \{-1,1\}$ are given. We also assume that $k$ is a positive switch and the numbers $\eta, \eps\in (0,1)$ and $q,\ell \in \mathbb{N}$ are chosen according to the following hierarchy:
 \begin{align}\label{eq: hierarchy} 
 0< \frac{1}{k} \ll \eta \ll \eps \ll \frac{1}{\ell} \ll \frac{1}{q} \ll \frac{1}{\|p\|}, \quad \frac{1}{d} \leq \frac{1}{2}.
 \end{align}

As $k$ is sufficiently large, we have that $m_s(k)>0$ for all $s\geq 1$. Recall that we have defined $(k,s)$-bad pair $\{a,b\}$ and Lemma~\ref{lem: bad} yields that many pairwise disjoint bad pairs provide many monochromatic solutions.

Recall that we have defined two numbers $k_1$ and $k_0$, and we further define $k_2$ as follows.
$$k_2:= \frac{1}{q} k, \enspace k_1 := \max\{ t: p(t)< \frac{1}{2}p(k))\} \enspace \text{and} \enspace k_{0} := \max\{ t: p(t) < 2p(k) - 4m_1(k)\}.$$ As $1/k \ll \eps$, we can estimate that $k_1 = (2^{-1/d} \pm \eps)k$ and $k_{0} = (2^{1/d} \pm \eps) k$.
As $k$ is sufficiently large, we may assume that $p(z)$ is a strictly increasing function on $[k_2,\infty)$.

The following lemma allows us to deal with the case where $k$ is a non-isolated switch.
\begin{lemma}[\cite{LPS}]\label{lem: non isolated}
If $k$ is a non-isolated switch, then the interval $[p(k)]$ contains at least $k^{1-o(1)}$ distinct monochromatic solutions.
\end{lemma}

As we will use this lemma to deal with the case when $k$ is a non-isolated switch,  we will deal with the case where $k$ is an isolated switch in the rest of this section.
Under this assumption, $a\in \Phi^+$ and $b\in \Phi^-$ such that the difference $b-a$ is $m_s(k)$ for some $s< k_0-k$ will yield a bad pair $\{a,b\}$.

\subsection{Utilizing nested pairs of almost monochromatic intervals}
In this subsection, we present a key lemma, Lemma~\ref{lem: many intervals}, which provides many monochromatic solutions utilizing a sequence of nested pairs of almost monochromatic intervals of opposite colors. Let us warm up with a single pair of intervals.

Assume that we have an interval $P$  mostly consisting of numbers from $\Phi^+$ and another interval $N$ mostly consisting of numbers from $\Phi^-$ with $\min(P)< \min(N)$. 
By expressing $\min(N)-\min(P)$ as an integral combination of the numbers $m_{s}(k)$ and considering a sequence as we explained in Section~\ref{sec: proof idea}, we may obtain many monochromatic solutions.
The following lemma proves this.

\begin{lemma}\label{lem: two intervals}
Suppose that $k$ is an isolated switch
and two intervals $P, N \subseteq [1, p(k)-k ]$ satisfy the following.
\begin{enumerate}[label={\rm (P\arabic*)}]
\item \label{P1} $ |P| = |N| = r < m_1(k)$.
\item \label{P2} $ \min(N)- \min(P) \geq \ell^2 k^{d-1/2}$.
\item \label{P3} $|P\cap \Phi^+| + |N \cap \Phi^-| \geq \frac{5}{4}r.$
\end{enumerate}
Then, there exists at least $r/4 - 2\ell$ distinct monochromatic solutions in $[p(k)]$.
\end{lemma}
\begin{proof}
As $k$ is an isolated switch, every pair $\{a,b\} \subseteq [p(k)]$ with $b=a+m_{s}(k)$, $s\in [k_0-k]$ and $\phi(a) > \phi(b)$ forms a bad pair. 
Let $T=\min(N)-\min(P)$.
Let $I:= \{ t 2^j :t\in [d-1], 1\leq j\leq q+ \frac{1}{2}\log{k}\}$ and 
$$M :=\sum_{i\in I } \ell \cdot m_{i}(k).$$
As it is easy to see $m_{i}(k)\leq 2^{2q} k^{d-1}$ for each $i\in I$ from \eqref{eq: mik},
this choice of $M$ yields $M < \ell^2 k^{d-1/2}.$
Let $s$ be the maximum non-negative integer such that 
$$s  m_1(k)  \leq T-M.$$
The maximality of $s$ ensures $T-M - s m_1(k) \leq m_1(k) \leq q k^{d-1}$. 
Hence, Proposition~\ref{lem: estimation} together with \eqref{eq: hierarchy} 
implies that there exists a vector $\mathbf{v}\in [-\ell,\ell]^{k}$ such that $\mathbf{m}^{k} \cdot \mathbf{v} = T-M -s m_1(k)\pm \ell$. This yields
\begin{align}\label{eq: target 1}
\mathbf{m}^{k} \cdot (\mathbf{v} + \ell \mathbf{1}_I ) +s m_1(k)= T  \pm \ell.
\end{align}
Note that we have $ \mathbf{v} + \ell \mathbf{1}_{I } \in [2\ell]^{k}$.
Let $w_1,\dots, w_{t}$ be an arbitrary sequence such that 
for each $i\in I$, the term $m_{i}(k)$ occurs exactly $\mathbf{v}_{i} + \ell$ times in the sequence and $m_1(k)$ occurs exactly $s$ times. Note that $1\notin I$.

For each $x\in P$ and $j\in [0,t]$, let 
$$u_j(x):= x+ \sum_{i\in [j]} w_{i}.$$
Together with \ref{P1} and \eqref{eq: target 1}, we conclude that
for all numbers $x\in [\min(P)+\ell, \max(P)-\ell] \cap \Phi^{+}$,
\begin{align}
u_t(x) \in N. 
\end{align}
However, \ref{P3} implies that there are at least $r/4 - 2\ell$ choices of such $x$ satisfying that $\phi(x)=+1$ and
$\phi(u_{t}(x) ) = -1$.
This means that 
the sequence $u_0(x),u_1(x),\dots, u_{t}(x)$ contains a bad pair. As two distinct choices of $x,x'$ gives us pairwise disjoint sequences, we can choose at least $r/4-2\ell$ pairwise disjoint bad pairs, hence Lemma~\ref{lem: bad} implies that there are at least $r/4- 2\ell$ distinct monochromatic solutions.
\end{proof}

We can use Lemma~\ref{lem: two intervals} to derive the following lemma for the case of $d=2$.

\begin{lemma}\label{lem: d=2}
Suppose that $d=2$ and $k$ is an isolated switch. Then $[p(k)]$ contains at least $\eps^2 k$ distinct monochromatic solutions.
\end{lemma}
\begin{proof}
As $k$ is an isolated switch, all numbers in $[k+1,k_0]$ are colored with $-1$.
We consider the following two cases. \newline

\noindent {\bf Case 1. There exists a number 
$x\in [k_2,(1-\eps) k_1]$ colored with $-1$.} Let $$P := p(x)-[k+1,k_{0}] \text{ and } N:= p(k)-P.$$
As all numbers in $[k+1,k_0]$ has color $-1$, 
if $y\in P=p(x)-[k+1,k_0]$ has the color $-1$, then $(y,p(x)-y, x)$ forms a monochromatic solution. Hence, we may assume that $P$ contains less than $\eps^2 k$ numbers of color $-1$.
In particular, 
$$|P\cap \Phi^{+}| \geq k_0-k - \eps^2 k \geq (1-\eps)|P|.$$
Then we conclude that 
$$|N\cap \Phi^{-}|\geq (1 -2\eps )|N|,$$
as otherwise we obtain at least $\eps |N|\geq \eps^2 k$ monochromatic solutions of the form $(y,p(k)-y,k)$ with $y\in P\cap \Phi^{+}$ and $p(k)-y\in N\cap \Phi^{+}$.  
Now we verify that these intervals $P$ and $N$ satisfy the conditions \ref{P1}--\ref{P3} of Lemma~\ref{lem: two intervals}
As
$P,N$ are two intervals of length $k_{0}-k$ and at least $(1-2\eps)$-fraction of the numbers in $P$ are colored $+1$ and at least $(1-2\eps)$-fraction of the numbers in $N$ are colored $-1$, \ref{P3} holds. Moreover, as $k$ is sufficiently large, the definition of $k_1$ implies 
$$p(x)\leq p((1-\eps)k_1) < (1-\eps)\frac{p(k)}{2},$$
hence 
$$P \subseteq \Big[ \frac{(1-\eps^2)p(k)}{2}\Big] \text{ and } N\subseteq  \Big[\frac{(1+\eps^2)p(k)}{2}, p(k) \Big].$$ 
As $\eps^2 p(k) > \ell^2 k^{d-1/2}$, \ref{P2} also holds.
As $|P|=|N| = k_0-k \leq m_1(k)$, \ref{P1} holds as well. Thus Lemma~\ref{lem: two intervals} implies that there are at least $\eps k$ distinct monochromatic solutions in $[p(k)]$. \newline

\noindent {\bf Case 2. All the numbers in $[k_2,(1-\eps) k_1]$ has color $+1$.} Let 
$$P':= [k_2,(1-\eps) k_1] \text{ and } N':= \Big[ \frac{p(k)-|P'|}{2}, \frac{p(k)+|P'|}{2} \Big].$$ Then, at least $0.4$-fraction of the numbers in $N'$ have color $-1$, as otherwise there are $\eps k$ monochromatic solutions of the form $(y,p(k)-y, k)$ where both $y$ and $p(k)-y$ are in $N'$ and colored $+1$. Hence, $P',N'$ are two intervals of length $(1-\eps)k_1 - k_2 \geq k/10$ which are at least $p(k)/4 - k_1 \geq \ell^2 k^{d-1/2}$ apart, thus 
\ref{P2} and \ref{P3} both hold. As $|P'|=|N'|\leq m_1(k)$, Lemma~\ref{lem: two intervals} implies that there are at least $
\eps k$ distinct monochromatic solutions in $[p(k)]$.
\end{proof}

The above lemma is sufficient for us to deal with the case $d=2$, but for larger values of $d$, it is not sufficient as we need to obtain more monochromatic solutions. 
For $d\geq 3$, we will take many interval pairs of opposite colors and repeatedly apply Lemma~\ref{lem: two intervals} to those interval pairs $(P_i,N_i)$. 
If the sets $[\min(P_i),\max(N_i)]$ form disjoint intervals, then repeated applications of Lemma~\ref{lem: two intervals} yield disjoint bad pairs.
However, we will only be able to obtain intervals in a nested way as follows. $$[\min(P_1),\max(N_1)]\subseteq [\min(P_2),\max(N_2)] \subseteq \dots \subseteq 
[\min(P_{\eta k}),\max(N_{\eta k})]$$
As these interval pairs will be nested, application of Lemma~\ref{lem: two intervals} to the pair $P_{i_*+1},N_{i_*+1}$ may yield some new bad pairs within $[\min(P_{i_*}),\max(N_{i_*})]$, which could intersect with old bad pairs obtained in previous pairs of intervals. Hence, we need to ensure that many of these new bad pairs are disjoint from the old bad pairs.

As explained in Section~\ref{sec: proof idea}, for each $x\in P$, we will choose a sequence from $x\in P_{i_*+1}$ to $y\in N_{i_*+1}$ which avoids old bad pairs.
Let $B_{i_*}$ be the set of all numbers in an old bad pair. 
We will restrict our choice of sequences to certain `good' sequences. Then for each $x\in P_{i_*+1}$, we consider the set $B^x$ of all possible numbers $b\in B_{i_*}$, which may be included in some `good' sequence from $x$.
By appropriately designing a `good' sequence to make it compatible with  Proposition~\ref{prop: t terms independent}, we can ensure that $B^x$ and $B^{x'}$ are pairwise disjoint for all $x\neq x'\in P_{i_*}$. As we have controls on the size of $B_{i_*}$, this disjointness allows us to choose a large subset $P'\subseteq P_{i_*}$ of numbers $x$ such that $B^x$ is small and  `thin'. 
Furthermore, exploiting the `goodness' of the sequence and using the fact that $B^x$ is small and `thin' for those $x\in P'$, we can transform this into a graph problem regarding paths and utilize Lemma~\ref{lem: grid} to find an appropriate sequence from $x$ to $y\in N_{i_*+1}$ which does not contain any numbers in $B_{i_*}$. 
The following lemma will encapsulate these ideas.

\begin{lemma}\label{lem: many intervals}
Suppose $d\geq 3$ and $k$ is an isolated positive switch and $Q$ is a number with $\frac{1}{\ell} p(k) \leq Q\leq p(k)$.
Suppose that the intervals $P_1,\dots, P_{\eta k}, N_1,\dots, N_{\eta k}$ in $[Q]$ satisfy the following for each $i\in [\eta k]$.
\begin{enumerate}[label={\rm (PN\arabic*)}]
\item \label{PN1} $\eps k < |P_i| = |N_i|  < k/2$ and $N_i = Q - P_i$.
\item \label{PN2} $\max(P_{i})- \max(P_{i+1}) \geq 10 \ell  m_1(k)$,
%\item\label{PN3}  $N_i = p(k) - P_i$.
\item \label{PN4} $\max(N_1) - \min(P_1) \geq \frac{1}{\ell} Q$.
\item \label{PN5} $|P_i\cap \Phi^+| + |N_i \cap \Phi^-| \geq \frac{5}{4}|P_i|.$
\end{enumerate}
Then, there exists at least $ \frac{\eta^3 k^{2}}{2\log{k}}$ distinct monochromatic solutions in $[p(k)]$.
\end{lemma}

As Lemma~\ref{lem: many intervals} is rather involved, we defer its proof to the end of this section. Lemma~\ref{lem: many intervals} can be used to prove Lemma~\ref{lem: isolated} below to handle all higher degree polynomials.

\begin{lemma}\label{lem: isolated}
If $k$ is an isolated positive switch, and $d\geq 3$, then $[k_2, p(k)]$ contains at least $\frac{\eta^3 k^{2}}{2\log{k}}$ distinct monochromatic solutions.
\end{lemma}
\begin{proof}
As $k$ is an isolated positive switch, every $\{a,b\} \subseteq [p(k)]$ with $b=m_{s}(k)$ for some $s\in [k_0-k]$ and $\phi(a) > \phi(b)$ forms a bad set. Assume for a contradiction that
\begin{align}\label{eq: assumption}
\text{there are less than  $\frac{\eta^3 k^{2}}{2\log{k}}$ monochromatic solutions.}
\end{align}
The following notations will be convenient for us. Let $k' := \min\{  [k_1,k] \cap \Phi^{-} \}$ be the smallest number of color $-1$ in the interval $[k_1,k]$ if exists. 
For each $i\in [k_2,k]$, let 
\begin{align}\label{eq: ABCD def}
\begin{split}
A_i &:= p(i) - [k+1,k_{0}],  \enspace
A'_i :=  p(k) - A_i, \\
B_i &:= p(i) - [k_2,k_1], \enspace
B'_i := p(k+1) - B_i \enspace \text{ and } \enspace
B''_i:= p(k') - B_i.
\end{split}
\end{align}

We consider the following three cases. \newline

\noindent {\bf Case 1.} $| \Phi^{-} \cap [k_2,k_1]| > 3\eps k$.

For each $i\in \Phi^{-}\cap [k_2,k_1]$, if $y\in A_i$ satisfies $\phi(y)=-1$, then $(y, p(i)-y, i)$ forms a monochromatic solution, because $p(i)-y \in [k+1,k_{0}] \subseteq \Phi^{-}$ is colored $-1$ as $k$ is an isolated switch.
Hence, if more than $\eps k$ elements $i\in \Phi^{-}\cap [k_2,k_1]$ satisfy $|A_i \cap \Phi^-| \geq 2\eps k$, then we obtain at least $2\eps^2 k^2$ distinct monochromatic solutions, a contradiction to \eqref{eq: assumption}.
Hence, there is $I\subseteq \Phi^{-}\cap [k_2,k_1] $ with 
\begin{align}\label{eq: A pos}
|I| \geq |\Phi^{-}\cap [k_2,k_1]| -\eps k\geq  2\eps k \enspace \text{such that} \enspace 
|A_i \cap \Phi^{-}| \leq 2\eps k \text{ for each $i\in I$}.
\end{align}
Again, for $i\in I$, if more than $3\eps k$ elements $y$ in $A'_i$ satisfy $\phi(y)=+1$, then there are at least $3\eps k - |A_i\cap \Phi^{-}|\geq \eps k$ choices of $y\in A'_i\cap \Phi^+$ satisfying $p(k)-y\in A_i \cap \Phi^+$ and forms a monochromatic solution $(y,p(k)-y,k)$.
Hence, by the same logic as above, in order to avoid obtaining $\eps^2 k^2$ monochromatic solutions, we conclude that there is $I'\subseteq I$ with
\begin{align}\label{eq: A' neg}
|I'| \geq |I|-\eps k \geq \eps k \enspace \text{such that} \enspace 
|A'_i \cap \Phi^{+}| \leq 3\eps k \text{ for each $i\in I'$}.
\end{align}
Let $i_{\eta k} < i_{\eta k-1}< \dots < i_{1}$ be the numbers in $I'$ which are smaller than $k_1- 2\eps k/3$
such that for each $j\in [\eta k]$, 
\begin{align}\label{eq: i scatter}
i_{j} > i_{j+1}+ \ell^2.
\end{align}
Indeed, such numbers exists as $\ell^2 \eta k < 2\eps k/3$. Moreover, as the polynomial $p$ has degree $d \geq 2$, we know $p(i_1)\leq p(k_1-2\eps k/3) = ((1-2\eps /3)^d + o(1)) p(k_1) < (1-\eps) p(k)/2$.
For each $j\in [\eta k]$, let 
$$P_j := A_{i_j} \enspace \text{and} \enspace 
N_j := A'_{i_{j}}.$$
By \eqref{eq: ABCD def}, this choice satisfies \ref{PN1} with $Q=p(k)$.
As $k_2>k/q$,
the equation \eqref{eq: i scatter}, \eqref{eq: mk property} and the definition \eqref{eq: ABCD def} of sets $A_i, A'_i$ imply that \ref{PN2} holds.
Also the fact that $p(i_1)\leq (1-\eps)p(k)$ implies \ref{PN4}, furthermore
\eqref{eq: A pos} and \eqref{eq: A' neg} imply \ref{PN5}. 
Hence, we can apply Lemma~\ref{lem: many intervals} to conclude that there are at least $\frac{\eta^3 k^{2} }{2\log{k}}$ distinct monochromatic solutions in this case, a contradiction to the assumption \eqref{eq: assumption}.
\newline

\noindent {\bf Case 2.} $| \Phi^{-} \cap [k_2,k_1]| \leq 3\eps k$ and 
$| \Phi^{+} \cap [k_1,k]| > 3\eps k$.

Consider a number $i\in \Phi^{+}\cap [k_1,k]$. If $|B_i\cap \Phi^{+}| \geq 4\eps k$, then we obtain at least $4\eps k -| \Phi^{-} \cap [k_2,k_1]|\geq \eps k$ distinct solutions of the form $(y,p(i)-y,i)$ where $y\in [k_2,k_1]\cap \Phi^+$ and $p(i)-y\in B_i \cap \Phi^{+}$.
Thus, by the assumption \eqref{eq: assumption}, at most $\eps k$ numbers $i$ in $\Phi^{+}\cap [k_1,k]$ satisfy $|B_i\cap \Phi^{+}| \geq 4\eps k$. 
This implies that there exists $J \subseteq \Phi^{+}\cap [k_1,k]$ such that
\begin{align}\label{eq: B neg}
|J| \geq | \Phi^{+} \cap [k_1,k]|-\eps k \geq   2\eps k \enspace \text{such that} \enspace 
|B_i \cap \Phi^{+}| \leq 4\eps k \text{ for each $i\in J$}.
\end{align}
Again, for each $i\in J$, if more than $5\eps k$ elements $y$ in $B'_i$ satisfy $\phi(y)=-1$, \eqref{eq: B neg} ensures   monochromatic solutions of the form $(y,p(k)-y,k+1)$ for at least $5\eps k - |B_i\cap \Phi^{+}| \geq \eps k$ choices of $y$. Hence, similarly as above, at least $\eps k$ choices of $i$ would contradict 
the assumption \eqref{eq: assumption}, we conclude that there exists a set $J'\subseteq J$ with 
\begin{align}\label{eq: B' pos}
|J'| \geq |J|-\eps k \geq \eps k \enspace \text{such that} \enspace 
|B'_i \cap \Phi^{-}| \leq 5\eps k \text{ for each $i\in J'$}.
\end{align}
Let $j_{1} < j_2 < \dots < j_{\eta k}$ be the numbers in $J'$ bigger than $k_1 + 2\eps k/3$ such that for each $i\in [\eta k]$, 
\begin{align}\label{eq: j scatter}
j_{i+1} > j_{i}+ \ell^2.
\end{align}
Indeed, such numbers exist as $\ell^2 \eta k < \eps k/3$.
Moreover, as the polynomial $p$ has degree $d\geq 2$, we have $p(j_1)> p(k_1+2\eps k/3) \geq  a_{d} (1+2d\eps/3)k_1^d- q \eps^2 k_1^{d} \geq \frac{1}{2}p(k_1) + \frac{3}{5}\eps p(k)$.

For each $i\in [\eta k]$, let 
$$P_i := B'_{j_i} \enspace \text{and} \enspace 
N_i := B_{j_{i}}.$$
By \eqref{eq: ABCD def}, this choice satisfies \ref{PN1} with $Q=p(k')$.
 By \eqref{eq: j scatter}, \eqref{eq: mk property} and the definition \eqref{eq: ABCD def} of sets $B_i, B'_i$, \ref{PN2} hold.
Also as $p(j_1) \geq  \frac{1}{2}p(k_1) + \frac{3}{5}\eps p(k)$, we also have \ref{PN4}. Finally, \eqref{eq: B neg} and \eqref{eq: B' pos} imply \ref{PN5}. 
Therefore, we can apply Lemma~\ref{lem: many intervals} to obtain at least $\frac{\eta^3 k^{2} }{2\log{k}}$ distinct monochromatic solutions, a contradiction to the assumption \eqref{eq: assumption}.\newline

\noindent {\bf Case 3.} $| \Phi^{-} \cap [k_2,k_1]| \leq 3\eps k$ and 
$| \Phi^{+} \cap [k_1,k]| \leq 3\eps k$.

In this final case, at most $3\eps k$ numbers in $[k_1,k]$ are colored $+1$, so $k'$ exists as in \eqref{eq: ABCD def} and it is between $k_1$ and $k_1 + 3\eps k+1$.

Consider a number $i\in \Phi^{+}\cap [k_2,k_1]$. 
If $|B_i\cap \Phi^{+}| \geq 4\eps k$, then we obtain at least $4\eps k- | \Phi^{-} \cap [k_2,k_1]|\geq  \eps k$ distinct solutions of the form $(y,p(i)-y,i)$ where $y\in [k_2,k_1]\cap \Phi^{+}$ and $p(i)-y\in B_i\cap \Phi^{+}$.
Hence, again by \eqref{eq: assumption}, there are at most $\eps k$ choices of such a number $i$.
Thus there exists $L \subseteq \Phi^{+}\cap [k_2,k_1]$ such that
\begin{align}\label{eq: B'' neg}
|L|\geq k_1-k_2 -| \Phi^{-} \cap [k_2,k_1]| -\eps k  \geq k_1-k_2- 4\eps k \enspace \text{and} \enspace 
|B_i \cap \Phi^{+}| < 4\eps k, ~~\forall i\in L.
\end{align}
Again, for a number $i\in L$, if more than $5\eps k$ elements $y$ in $B''_i$ satisfy $\phi(y)=-1$, then $(y,p(k')-y,k')$ forms a monochromatic solution for at least $5\eps k -|B_i \cap \Phi^{+}| \geq \eps k$ choices of $y\in B_i\cap \Phi^{-}$ and $p(k')-y\in B''_i\cap \Phi^{-}$. 
Hence, by the same logic as above, \eqref{eq: assumption} implies that there are at most $\eps k$ choices of such $i$. Thus there is a set $L'\subseteq L$ with
\begin{align}\label{eq: B'' pos}
|L'| \geq |L|-\eps k\geq k_1-k_2-5\eps k \enspace \text{such that} \enspace 
|B''_i \cap \Phi^{-}| \leq 5\eps k \text{ for each $i\in L'$}.
\end{align}
Let $b_1 < b_2  < \dots < b_{\eta k}$ be the numbers in $L'$ which are bigger than $k_1 - 6\eps k$
such that for each $i \in [\eta k]$, 
\begin{align}\label{eq: b scatter}
b_{i +1}> b_{i}+ \ell^2.
\end{align}
Indeed, such numbers exist as at least 
$\eps k > \ell^2 \eta k$ numbers in $L'$ are bigger than $k_1-6\eps k$.
Moreover, as $\eps\ll 1/d$ and $k'\leq k_1+3\eps k$, we have 
$p(b_1) > p(k' - 9\eps k) \geq  \frac{2}{3}p(k')$.
For each $i\in [\eta k]$, let 
$$P_i := B''_{b_i} \enspace \text{and} \enspace 
N_i := B_{b_{i}}.$$
By the definitions in \eqref{eq: ABCD def}, the condition \ref{PN1} is satisfied with $Q=p(k')$. As $k> k_1$, $\frac{1}{2}p(k)\leq Q \leq p(k)$ holds.
 By \eqref{eq: b scatter} and \eqref{eq: mk property}, \ref{PN2} holds.
Moreover, as we have $p(b_1) - (p(k')-p(b_{\eta k})) \geq  \frac{1}{3} p(k')$, \ref{PN4} holds and
\eqref{eq: B'' neg} and \eqref{eq: B'' pos} imply \ref{PN5}. 
Hence, we can apply Lemma~\ref{lem: many intervals} to obtain at least $\frac{\eta^3 k^{2} }{2\log{k}}$ distinct monochromatic solutions, a contradiction to \eqref{eq: assumption}.
As we have exhausted all three cases, this  proves the lemma.
\end{proof}

\subsection{Putting things together}
We need one more lemma for the proof of Theorem~\ref{thm: counting}. As we assume that $k$ is an isolated switch, all the numbers in $[k+1,k_0]$ are colored $-1$. These consecutive numbers of the same color was very useful for obtaining monochromatic solutions in the lemmas before.
In fact, if we have a much longer such interval $[k+1,r]$, then we can utilize this to obtain even more solutions as follows.

\begin{lemma}\label{lem: long minus}
Suppose that $k$ is a positive switch and $10q m_1(k) \leq r \leq \frac{1}{2} p(k_2)$ and $[k+1, r] \subseteq \Phi^{-}$. Then $[p(k)]$ contains at least $ \frac{ \eta^{2} r k}{ m_1(k)}$ distinct monochromatic solutions.
\end{lemma}
\begin{proof}
We consider the following two cases depending on the colors of the numbers in the smaller interval $[k_2,k_1]$. \newline

\noindent {\bf Case 1. There exists $x\in [k_2,k_1]$ with color $-1$.}
For this choice of $x$, as $x\geq k_2$, we have $$r\leq \frac{1}{2} p(k_2) \leq  \frac{1}{2} p(x).$$
 Consider two intervals 
$$P:= p(x)- [k+1,r] \text{ and } N: = p(k)-P.$$
As $r< p(x)/2$ and $p(x)\leq p(k_1)\leq \frac{1}{2}p(k)$, we have $P \subseteq [p(k_1)-k-1]\subseteq [\frac{1}{2}p(k)-k-1]$.
Each  $y \in P$ with color $-1$ yields a monochromatic solution of the form $(y, p(x)-y, x)$ because $p(x)-y  \in [k+1,r]$ is colored with $-1$. Thus we may assume that at least $(1-\eps)$-fraction of the numbers in $P$ have color $+1$, otherwise we already obtain at least $\eps r\geq \frac{\eta^{2} r k }{ m_1(k) }$ distinct monochromatic solutions as $d\geq 2$.

Moreover, at least $(1-2\eps)$-fraction of the number in $N$ has color $-1$, as otherwise we obtain at least $\eps (r-k)\geq \frac{\eta^{2} r k }{ m_1(k) }$ monochromatic solutions of the form $(y, p(k)-y, k)$ where both $y\in P$ and $p(k)-y\in N$ are colored with $+1$.

Let $T:= \min(N)- \min(P)$, then  as $P\subseteq [\frac{1}{2}p(k)-k-1]$ we have $r \leq T\leq p(k)$.
As $k$ is sufficiently large, 
$m_1(k)< m_2(k) < \dots < m_{k_{0}}(k)$  is an increasing sequence.
Furthermore, from 
\eqref{eq: mik}, we can deduce that for each $i\in [k_0-k]$,
$$m_{i+1}(k)-m_i(k) = a_d dk^{d-1}+ O(k^{d-2}) \text{ and } m_1(k) = a_d dk^{d-1} + O(k^{d-2}).$$ 
As $k$ is sufficiently large, this implies $m_{i+1}(k) -m_i(k) \leq  q m_1(k)$ for all $i\in [k_0-k]$. 
Also, the definition of $k_0$ yields that $$ m_{k_0-k}(k)= p(k_0) - p(k) \geq p(k) - q m_1(k).$$
From these facts, we conclude that the union of intervals $\bigcup_{s\in [k_0-k]}[m_{s}(k)-qm_1(k), m_{s}(k)+qm_1(k)]$ covers the interval $[p(k)]$, hence there exists $s \in [k_{0}-k]$ such that $m_{s}(k) = T \pm q m_1(k)$.

For this choice of $s$, we have $|(m_{s}(k) + P) \cap N | \geq (r -k)- q m_1(k)\geq (\frac{1}{2}+4\eps) r$.
As $(1-2\eps)$-fraction of $P$ has color $+1$ and $(1-2\eps)$-fraction of $N$ has color $-1$, we obtain at least $(r-k)/2$ many $(k,s)$-bad sets $\{a,b\}$ with $a\in P$ and $b\in N$. Hence, Lemma~\ref{lem: bad} implies that there are at least $\frac{r-k}{2} \geq \frac{\eta^{2} r k }{ m_1(k) }$ distinct monochromatic solutions. \newline

\noindent {\bf Case 2. All the numbers in $[k_2,k_1]$ have the color $+1$.}
Let $$x:= \min\{ y \in [k_1+1, k+1] : \phi(y)=-1\} -1.$$ As $\phi(k+1) =-1$, the integer $x$ must exist and every $i\in [k_2,x]$ is colored with the color $+1$.

Let $I$ be the set of those numbers $i\in [k_2,x]$ for which the interval 
$p(i)- [k_2,k_1]$ contains at least $\eps k$ numbers of color $+1$.
Note that for each $i\in I$ and the number $y\in p(i)- [k_2,k_1]$ of color $+1$, we obtain a monochromatic solution $(y,p(i)-y,i)$.  As this yields at least $|I| \times \eps k$ distinct monochromatic solutions, we may assume that we have $|I| \leq  \frac{ \eta r  }{ m_1(k) }$.
Hence, there are at least $\frac{r}{4m_1(k)}$ numbers between $k-\frac{r}{2m_1(k)}$ and $k$ that does not belong to $I$. Let $J$ be the set of these numbers.
Consider a number $j\in J$.
As $r\leq \frac{1}{2}p(k_2)$, using \eqref{eq: mik} we can check that $k-j\leq q\eta k$. Thus using \eqref{eq: mik}, we have
$$k \leq  p(k+1)-p(j) = a_d d(k-j+1 \pm \ell \eta)k^{d-1}  
\leq (k-j+2) m_1(k).$$
This implies
$$p(j) - [k_2,k_1]  \subseteq p(k+1) - [k, k_1+ (k-j+2)m_1(k)]  \subseteq p(k+1) - [k, r].$$
Here, the final containment holds as $k- \frac{r}{2m_1(k)}\leq j \leq k$ and $r\geq 10 q m_1(k)$.
Thus for each $y\in p(j)-[k_2,k_1]$ with $\phi(y)=-1$, we have $p(k+1)-y \in [k,r]\subseteq \Phi^{-1}$, hence 
 $(y, p(k+1)-y, k+1)$ forms a monochromatic solution of color $-1$.
As $j\notin I$, there are at least $k_1-k_2-\eps k$ such choices of $y$ for each $j\in J$.
Therefore, we obtain at least
$$ \frac{r}{2m_1(k)} (k_1-k_2 -\eps k) \geq \frac{ \eta^2 r  k }{ m_1(k) }$$
distinct monochromatic solutions.
This proves the lemma.
\end{proof}

Now, we have collected all the lemmas and we are ready for the proof of Theorem~\ref{thm: counting}.

\begin{proof}[Proof of Theorem~\ref{thm: counting}]

Fix a number $q$ with $0<1/n\ll 1/q \ll 1/d, 1/\|p\|$ and two constants $c_1, c_2< 1$ with $1/q< a_d (10 c_1)^d< c_2 < 1/(10a_d)$ and assume $n$ is sufficiently large.
We consider an interval $I=[  c_1 n^{1/d^2},  c_2 n^{1/d}]$. As $n$ is sufficiently large, the polynomial $p(z)$ is strictly increasing  for $z\geq  c_1 n^{1/d^2}$. Moreover, by the choice of $c_1,c_2$, we have
\begin{align*}
    2 p(2c_1 n^{1/d^2}) < c_2 n^{1/d} \quad \text{ and }\quad 
2 p(c_2 n^{1/d}) < n.
\end{align*}
Thus if $I$ is monochromatic, we obtain at least $\Omega(n^{2/d^2})$ monochromatic solutions of the form $(x,y,z)$ with $z\in [c_1 n^{1/d^2}, 2c_1 n^{1/d^2}]$ and $x,y\in I$ as desired.
Otherwise, take the largest switch $k\in I$. 

If $k$ is not an isolated switch, there must be another switch $k^*$ larger than $k$ and $p(k^*)< 2p(k)-4m_1(k)$ by the definition of an isolated switch. In particular, $k^*< 2k$.
But the maximality of $k$ implies that such a switch must not be in $I$. Hence, $c_2 n^{1/d}< k^* < 2k$, meaning that $k> \frac{1}{2} c_2 n^{1/d}$. Now Lemma~\ref{lem: non isolated} implies that there are at least $k^{1-o(1)} = \Omega(n^{1/d-o(1)}) > n^{2/d^2-o(1)}$ distinct monochromatic solutions. Thus we may assume that $k$ is an isolated switch. \newline

{\noindent \bf  Case 1.} $d\geq 3$.
As $k$ is an isolated switch,  Lemma~\ref{lem: isolated} implies that there are at least $k^{2-o(1)}\geq n^{2/d^2-o(1)}$ distinct monochromatic solutions.\newline

{\noindent \bf  Case 2.} $d= 2$.
As $k$ is an isolated switch, all numbers in $I$ bigger than $k$ are colored $-1$.

If $c_2 n^{1/d} \leq 10 q m_1(k)$, then as $m_1(k) \leq d a_d k+ q$, we have $k\geq q^{-2} n^{1/d}$. Hence Lemma~\ref{lem: d=2} implies that there are at least $\Omega(k)=\Omega(n^{1/d}) \geq  \Omega(n^{2/d^2})$ distinct monochromatic solutions.

If $c_2 n^{1/d} \geq 10 q m_1(k)$, then
$m:=\frac{1}{2}p(\frac{1}{q} k) \leq c_2 n^{1/d}$ as $1/q < c_2$.
Moreover, $m= \Omega(k^2) = \Omega(n^{2/d^2})$ as $k>c_1 n^{1/d^2}$.
Thus Lemma~\ref{lem: long minus} with $m$ playing the role of $r$ yields at least $ \Omega(\frac{m k}{m_1(k)})= \Omega(m)
= \Omega(n^{2/d^2})$ distinct monochromatic solutions.
This proves the theorem.
\end{proof}

\subsection{Proof of the key lemma, Lemma~\ref{lem: many intervals}}
As $k$ is an isolated positive switch, every pair $\{a,b\} \subseteq [p(k)]$ with $b=a+m_{s}(k)$ for some $s\in [k_0-k]$ and $\phi(a) > \phi(b)$ forms a bad pair. 
We will inductively construct pairwise disjoint bad sets.
For each $i \in [\eta n-1]$, let 
\begin{align*}
S_i&:= [\min(P_{i+1}), \min(P_{i})-1], &T_i&:= [\max(N_{i})+1, \max(N_{i+1})]  \text{ and } \\
R&:= [\min(P_1), \max(N_1)]. & &
	\end{align*}
These sets provide a partition of the intervals $[\min(P_{\eta k}), \max(N_{\eta k})]$.

Assume that for some $i_* \in \{0,1,\dots, \eta k\}$, 
we have a collection $\mathcal{B}_{i_*}$ of pairwise disjoint bad pairs with $B_{i_*} := \bigcup_{\{a,b\} \in \mathcal{B}_{i_*}} \{a,b\}$ satisfying the following.
\begin{enumerate}[label={\rm (Ind\arabic*)}]
\item \hspace{-0.16cm}$_{i_*}$ \label{Ind1} $\displaystyle B_{i_*}\subseteq R\cup \bigcup_{i\in [i_*]} (S_i\cup T_i) \enspace \text{and} \enspace |B_{i_*}| = \frac{\eta^2 k}{\log {k}} \cdot \max\{ i_* - \eta k/2 , 0 \}. $

\item \hspace{-0.16cm}$_{i_*}$\label{Ind2} For each $j\in [i]$, we have
$\displaystyle \Big|B_{i_*}\cap  \bigcup_{i={i_*}-j+1}^{i_*-1} (S_i\cup T_i) \Big| \leq \frac{(j-1)\eta^2 k }{\log{k}}.$
\end{enumerate}

 Note that the empty collection $\mathcal{B}_i = \emptyset$ satisfies both \ref{Ind1}$_{i_*}$ and \ref{Ind2}$_{i_*}$ for $i_*\leq \eta k/2$. We assume that we have a collection $\mathcal{B}_{i_*}$ satisfying \ref{Ind1}$_{i_*}$ and \ref{Ind2}$_{i_*}$ with
$\eta k/2\leq i_*\leq \eta k$ and 
we will show that we can construct $\mathcal{B}_{i_*+1}$ satisfying \ref{Ind1}$_{i_*+1}$ and \ref{Ind2}$_{i_*+1}$
from the given $\mathcal{B}_{i_*}$. This shows that we can inductively build the collections $\mathcal{B}_{1},\dots, \mathcal{B}_{\eta k}$ and the final set $\mathcal{B}_{\eta k}$ will be our desired collection of pairwise disjoint bad pairs.

As explained before, we plan to find a sequence $\mathbf{u}(x) = u_0, u_1,\dots, u_{m}$ of copies of $m_s(k)$s
for some $x\in P_{i_*+1}\cap \Phi^+$ where $x$ and $x+ \sum_{i\in [m]} u_{i}\in N_{i_*+1}\cap \Phi^-$.
This will yield a bad pair. In order to make sure that this bad pair is disjoint from the ones in $\mathcal{B}_{i_*}$, we will carefully choose this sequence.
First, we need to decide how many $m_s(k)$ we have to put in $\mathbf{u}$ to ensure $x+ \sum_{i\in [m]} u_{i} \in N_{i_*+1}$.

Let $T := \min(N_{i_*+1})-\min(P_{i_*+1})$, then $T$ is our target value for $\sum_{i\in [m]} u_{i}$.
Note that $T$ is a large number with $Q/(2\ell)\leq T \leq Q$ from \ref{PN1} and \ref{PN4}.
Furthermore, from \ref{PN1}, we have  
\begin{align}\label{eq: min+T}
    \min(P_{i_*+1})+  \frac{1}{2}|P_{i_*+1}| + \frac{1}{2}T  = \frac{1}{2}Q.
\end{align}
 Let 
$$I:= \{ t 2^j: t\in [d-1], 1\leq j\leq q+ \frac{1}{2} \log{k}\} \text{ and }
M :=\sum_{i \in I } \ell ~ m_{i}(k).$$
Any number $i\in I$ satisfies $i\leq 2^{2q} k^{1/2}$ and \eqref{eq: mk property} implies that $\frac{1}{2}ia_d dk^{d-1}\leq m_{i}(k) \leq i q k^{d-1}$. Hence we have \begin{align}\label{eq: M size}
 \ell k^{d-1/2} < M < \ell^2 k^{d-1/2} \leq \eta Q.
\end{align}
Let $t$ be the maximum number such that 
$$ t (m_1(k)+m_2(k))  \leq T-M \leq T-\eta Q.$$
As $m_1(k)+m_2(k) <\eta Q$, the maximality of $t$ implies 
\begin{align}\label{eq: t'mm size}
    t (m_1(k)+m_2(k))  =  T\pm \eta Q.
\end{align}
Furthermore, as $Q/(2\ell)\leq T \leq p(k)\leq k m_1(k)$ and $Q> p(k)/\ell$, this implies 
\begin{align}\label{eq: s size}
 k/\ell^3 \leq t \leq k. 
\end{align}
By the maximality of $t$,  we have
$$T-M - t \sum_{s\in [2]} m_s(k) \leq \sum_{s \in [2]} m_s(k) \leq q k^{d-1},$$
so Proposition~\ref{lem: estimation} implies that 
there exists a vector $\mathbf{v} \in [-\ell,\ell]^k$ satisfying
$\mathbf{m}^{k}\cdot \mathbf{v} 
= T - M - t\sum_{s\in [2]} m_s(k) \pm \ell$ with $\mathbf{v}_i=0$ for all $i\notin I$.
This yields the following non-negative integral combination of $m_s(k)$ as desired.
\begin{align}\label{eq: target}
\mathbf{m}^{k} \cdot( \mathbf{v} + \ell \mathbf{1}_I ) + t\sum_{s\in [2]} m_s(k)= T \pm \ell.
\end{align}
We know that $ \mathbf{v} + \ell \mathbf{1}_{I } \in [2\ell]^{k}$.
We fix an arbitrary sequence $w_1,\dots, w_{t'}$ such that, for each $s\in I$,
$m_{s}(k)$ occurs exactly $\mathbf{v}_{s} + \ell$ times in the sequence and no other terms appear in the sequence. 
As the $i$-th coordinate of $\mathbf{v} + \ell \mathbf{1}_{I }$ is zero for all $i\notin I$,
the definition of $I$ yields
\begin{align}\label{eq: f' size}
 t'=\|  \mathbf{v} + \ell \mathbf{1}_{I} \|_1 \leq 2\ell |I| \leq 3d \ell  \log{k}.
\end{align}
For each $s\in [t']$, let 
$$w'_s= \sum_{i\in [s]} w_i \enspace \text{and} \enspace W'=\{ w'_s: s\in [0,t']\}.$$
In particular, $W'$ contains zero as $w'_0$ is zero.
As $m_s(k)\leq \ell k^{d-1/2}$ for all $s\in I$,
\eqref{eq: f' size} implies
\begin{align}\label{eq: w size}
  w'_{t'} \leq \ell^3 k^{d-1/2} \log k \leq \eta Q.
\end{align}

For each $x\in P_{i_*+1}$, let 
$$y(x):= x+ \mathbf{m}^{k} \cdot( \mathbf{v} + \ell \mathbf{1}_I ) + t\sum_{s\in [2]} m_s(k).$$
As \eqref{eq: target} ensures that this value is 
$x+ T \pm \ell$, we have $y(x)\in N_{i_*+1}$ for all 
$x\in P_{i_*+1}$ with $\min(P_{i_*+1})+\ell \leq x\leq \max(P_{i_*+1})-\ell$.
Let 
$$P:= \{ x\in P_{i_*+1}: y(x)\in N_{i_*+1}, \phi(x)=+1 \text{ and } \phi(y(x))=-1 \}.$$
Then \ref{PN5} and \eqref{eq: target} implies
$$|P| \geq \frac{1}{4}|P_{i_*+1}| - 2\ell \geq \frac{1}{5}\eps k.$$

For each $x\in P$, we aim to construct a sequence $\mathbf{r}(x)=r_1,\dots, r_{2t}$ which contains exactly $t$ copies of $m_1(k)$ and exactly $t$ copies of $m_2(k)$. After obtaining this, we will insert the terms $w_1,\dots, w_{t'}$ in the middle of this sequence as follows to obtain our desired sequence $\mathbf{u}(x)$
$$\mathbf{u}(x):=(u_1,\dots, u_{2t+t'}):= (r_1, r_2,\dots, r_{t}, w_1,\dots, w_{t'}, r_{t+1}, \dots r_{2t}).$$ 
With this choice, the sequence
$$x+ \mathbf{u}(x):= (x, x+u_1, x+u_1+u_2,\dots, x+ \sum_{i\in [2t+t']} u_i )$$
will be our final desired sequence which starts at $x\in \Phi^+$ and ends at $y(x)\in \Phi^-$ and all consecutive terms differ by $m_s(k)$ for some $s\in I\cup \{1,2\} \subseteq [k_0-k]$.

We will make sure that
the sequence $\mathbf{r}(x)$ chosen for this $x$ is constructed so that the above sequence $x+ \mathbf{u}(x)$ does not intersect with $B_{i_*}$ for all $x\in P'$ where $P'$ is a not so small subset of $P$.

We call a sequence $\mathbf{r}=(r_1,\dots, r_{2t})$ \emph{admissible} if the number of occurrences of $m_1(k)$ and $m_2(k)$ in the first $j$ terms differ by at most $\eta^2 t$ for all $j\in [2t]$.
 As admissible sequences are nice to work with, we will choose our sequence $\mathbf{r}(x)$ among the admissible sequences.

If a sequence $r_1,\dots, r_{2t}$ is admissible, then \eqref{eq: t'mm size} and \eqref{eq: min+T} imply the following.
\begin{align}\label{eq: admissible}
    \begin{split}
        \sum_{i\in [t]} r_i  &= (\frac{1}{2} \pm \eta^2 )t\sum_{s\in [2]}m_{s}(k) =  \frac{1}{2} (T\pm 2\eta Q) \pm  \eta Q =  \frac{1}{2}Q - \min(P_{i_*+1}) \pm 4\eta Q \\
    &= (\frac{1}{2}\pm 4\eta)Q- \min(P_{i_*+1}).
    \end{split}
\end{align}
Our desired sequence $\mathbf{u}(x)$ is different from $\mathbf{r}(x)$ as we insert the terms $w_1,\dots, w_{t'}$ in the middle of the sequence. As the sequence $\mathbf{r}(x)$ consists of only $m_1(k)$ and $m_2(k)$,  it is simpler. So we want to focus on the sequence $\mathbf{r}(x)$. We define the following set $B$ so that we only have to focus on avoiding $B$ in the sequence 
$$x+\mathbf{r}(x) := (x, x+r_1, x+r_1+r_2,\dots, x+ \sum_{i\in [2t]} r_i), $$ instead of worrying about whether the sequence $x+\mathbf{u}(x)$ intersects $B_{i_*}$. Let
\begin{align*}
  B&:= B_{\rm down}\cup  B_{\rm mid} \cup B_{\rm up} \enspace \text{ with } \\
  B_{\rm down} &:= \{ b\in B_{i_*} : b< (\frac{1}{2}-\eps)Q\}, \\
  B_{\rm mid} &:= \{ b\in B_{i_*} :  (\frac{1}{2}-\eps)Q\leq b \leq (\frac{1}{2}+\eps)Q\} - W',\\
  B_{\rm up} &:= \{ b\in B_{i_*} :   b > (\frac{1}{2}+\eps)Q\} - w'_{t'}.
\end{align*}
In other words, $B_{\rm down}$ consists of all the numbers in $B_{i_*}$ which are smaller than $(\frac{1}{2}-\eps)Q$, and $B_{\rm mid}$ consists of all numbers of the form $b-w'_s$ for some $b\in B_{i_*}$ with $b=(\frac{1}{2}\pm \eps)Q$ and $s\in [0,t']$,
and $B_{\rm up}$ consists of all numbers of the form $b - w'_{t'}$ for some $b\in B_{i_*}$ with $b> (\frac{1}{2}+\eps)Q$.

The following claim shows that this set is indeed the set that we need to avoid intersecting with $x+\mathbf{r}(x)$.
\begin{claim}\label{cl: r useful}
    For a number $x\in P$ and an admissible sequence $\mathbf{r}(x) = r_1,\dots, r_{2t}$, if $x+ \sum_{i\in [s]} r_i \notin B$ for all $s\in [2t]$, then $x+ \sum_{i\in [s]} u_i \notin B_{i_*}$ for any $s\in [2t+t']$.
\end{claim}
\begin{proof}
Suppose not. Then there exists $b\in B_{i_*}$ such that $ b = x+ \sum_{i\in [s]} u_i$
for some $s\in [2t+t']$.

If $b< (\frac{1}{2}-\eps)Q$, then $b\in B_{i_*}$ implies that $b\in B_{\rm down}$.
Using \eqref{eq: admissible} yields 
$$b<(\frac{1}{2}-\eps)Q< \min(P_{i_*+1})  +\frac{1}{2}Q - \min(P_{i_*+1}) - 5 \eta Q  \leq  x+ \sum_{i\in [t]} r_i.$$
Hence, $s$ must be at most $t$ in this case and
this yields  
$$x+ \sum_{i\in [s]} u_i = b \in B_{\rm down} \subseteq B,$$
which contradicts the assumption of the claim.

If $b> (\frac{1}{2}+\eps)Q$, then \eqref{eq: w size} and \eqref{eq: admissible} with $c=t'$ ensures that 
$$b\geq (\frac{1}{2}+\eps)Q> \min(P_{i_*+1}) + k+ \frac{1}{2}Q + 5 \eta Q -\min(P_{i_*+1}) \geq  x+ \sum_{i\in [t]} r_i + w'_{t'}.$$
Hence we have $s\geq t+t'$.
As $b\in B_{i_*}$ and $b>(\frac{1}{2}+\eps)Q$, we have $b-w'_{t'}\in B_{\rm up}$. As $s\geq t+t'$, we have
$$ x+ \sum_{i\in [s-t]} r_i=b-w'_{t'}   \in B_{\rm up} \subseteq B,$$ which  contradicts the assumption of the claim.

Finally, consider the case of $b= (\frac{1}{2} \pm \eps)Q$. 
If $s\leq t$, then we have $x+ \sum_{i\in [s]} r_i = b\in (B_{i_*}\cap U)
\subseteq B_{\rm mid}$ as $0\in W'$.
If $s\geq t+t'$, then we have 
$b-w'_{t'} \in B_{i_*}-w'_{t'} \subseteq B_{\rm mid}$ as $w'_{t'}\in W'$.
Thus $ x+ \sum_{i\in [s-t']} r_i= b-w'_{t'} \in B_{\rm mid}\subseteq B$, a contradiction to the assumption of the claim. 

If $t<s\leq  t+t'$, then 
$w'_{s-t} = \sum_{i\in [s-t]} u_{t+i} \in W'$, 
thus
$$x+ \sum_{i\in [t]} r_{i}  = b- w'_{s-t} \in B_{\rm mid}\subseteq B,$$
which contradicts the assumption of the claim. This proves the claim.
\end{proof}
This claim ensures that we only have to worry about whether the sequence $x+\mathbf{r}(x)$ hits $B$ or not.
For each $x\in P$, we consider the following reduced grid digraph $G_x$ isomorphic to $G(t;\eta^2 t)$ with the vertex set $V_x$ where
\begin{align*}
V_x &:= \{ (x,s_1,s_2)\in P\times [t]^2 :  |s_1-s_2|\leq \eta^2 t\} \text{ and } V=\bigcup_{x\in P'}V_x.
\end{align*}
Consider a map $\alpha: V\rightarrow [Q]$ such that $\alpha(x,s_1,s_2) = x+ s_1 m_1(k)+ s_2m_2(k)$. 
This map $\alpha$ sends each vertex to a number of the form $x+ s_1m_1(k)+ s_2m_2(k)$, thus
a directed path in the grid graph $G_x$ naturally corresponds to a sequence of the form $x+ \mathbf{r}(x)$. We claim that $\alpha$ is an injective function.
\begin{claim}
    $\alpha$ is an injective function.
\end{claim}
\begin{proof}
Suppose not, then there exist $(x,s_1,s_2)\neq (x',s'_1,s'_2)\in V$ and $b\in [Q]$ with 
$$ b= x+s_1 m_1(k) + s_2 m_2(k) = x' + s'_1 m_1(k) + s'_2 m_2(k).$$
Assume $x\geq x'$.
By the definition of $V$, we have $s_2 = s_1 \pm \eta^2 t$, so
$$ (s_1-\eta^2 t)(m_1(k)+ m_2(k))\leq b-x \leq b-x' \leq (s'_1+\eta^2 t) (m_1(k)+ m_2(k)).$$
This implies that $|s_1-s'_1|\leq 2\eta^2 t$ and similarly we obtain $|s_2-s'_2|\leq 2\eta^2 t$.
As
$$x-x' = (s'_1-s_1) m_1(k) + (s'_2-s_2)m_2(k),$$
Proposition~\ref{prop: t terms independent} with $d=3$ and $t=2$ implies $s'_1-s_1= s'_2-s_2 =0$ or $|x-x'|>k/2$. As both $x$ and $x'$ are in $P$, \ref{PN1} implies $|x-x'|< k/2$, thus we must have $(s_1,s_2)=(s'_1,s'_2)$.
However, in such a case, we obtain $x-x'=0$, thus $(x,s_1,s_2)=(x',s'_1,s'_2)$. This proves that $\alpha$ is an injective function.
\end{proof}
As $\alpha$ is an injective function, the sets $\alpha(V_x)$ are pairwise disjoint over the choices of $x\in P$.
Note that $\alpha(V_x)$ is the set of all numbers in $[Q]$ which may belong to the sequence $x+\mathbf{r}(x)$ for an admissible sequence $\mathbf{r}(x)$. 

For each $x\in P$, let $L_x(a,b)$ be the set of vertices at level $j$ for some $a\leq j\leq b$ in the digraph $G_x$. 
Now, based on our choice of $B$, we can prove the following bounds on how it is distributed over the level sets in the graphs $G_x$.

\begin{claim}\label{cl: B distribution}
We have $|B|\leq  \ell^2 \eta^3 k^2$. Moreover, for each $i\in [\eta^2 t/\ell]$, 
    $\bigcup_{x\in P} \alpha(L_x(0,i\ell))$ and $\bigcup_{x\in P} \alpha(L_x(2t-i\ell,2t))$
    contains at most $\frac{2(i-1)\eta^2 k}{\log k}$ numbers in $B$.
\end{claim}
\begin{proof}
By \ref{Ind1}$_{i_*}$ and the definition of $B$, \eqref{eq: f' size}  implies
$$|B|\leq (|W'|+1)|B_{i_*}|\leq (1+3d\ell \log k ) \frac{\eta^3 k}{2\log k} \leq \ell^2 \eta^3 k^2.$$

For all $i\in [\eta^2 t/\ell]$, $x\in P$ and $v\in L_x(0,i\ell)$, we have 
$$\alpha(v) \leq x+ 2 i\ell  m_2(k) \leq \max(P_{i_*+1}) + 5 i\ell  m_1(k),$$
hence $i_*\geq \eta k/2>\eta^2 k$ and \ref{PN2} ensure that
$\alpha(v)$ belongs to $\bigcup_{j=i_*-i+1}^{i_*-1}S_{j}$ and all such numbers are  smaller than $(\frac{1}{2}-\eps)Q$ by \ref{PN4}.
Thus we conclude that 
$$\bigcup_{x\in P}\alpha( L_x(0,i\ell))\subseteq \bigcup_{j=i_*-i+1}^{i_*-1} S_j
.$$
As $B\cap \bigcup_{j=i_*-i+1}^{i_*-1} S_j$ belongs to $B_{\rm down}\subseteq B_{i_*}$ and $\alpha$ is injective, \ref{Ind2}$_{i_*}$ implies that the set
$\alpha(\bigcup_{x\in P} L_x(0,i\ell))$ contains at most $ \frac{(i-1)\eta^2 k}{\log k}$ numbers in $B$.

Similarly, we can show that for all $i\in [\eta^2 t/\ell]$, $x\in P$ and $v\in L_x(2t-i\ell,2t)$, we have 
$$\alpha(v) \geq y(x) - 2 \ell i m_2(k) \geq \min(N_{i_*+1}) - 5\ell i m_1(k).$$
By the definition of $\mathbf{u}(x)$ and the fact that $y(x)\in N_{i_*+1}$, \ref{PN2} implies that 
$\bigcup_{x\in P} \alpha(L_x(2t-i\ell,2t))$ belong to $\left(\bigcup_{j=i_*-i+1}^{i_*-1} T_j\right) - w'_{t'}$ and all numbers in $\left(\bigcup_{j=i_*-i+1}^{i_*-1} T_j\right)$ are larger than $(\frac{1}{2}+\eps)Q$. So, $\alpha(v)$ is never in $B_{\rm down}$ or $B_{\rm mid}$.
Applying \ref{Ind2}$_{i_*}$ again yields that
$\bigcup_{x\in P} \alpha(L_x(2t-i\ell,2t))$ contains at most $ \frac{(i-1)\eta^2 k}{\log k}$ numbers of $B_{\rm up}$. This proves the lemma. 
\end{proof}

Now, we consider the following pairwise disjoint sets over the choices of $x\in P$.
$$B^x = B\cap \alpha(V_x).$$
Let $P'$ be the set of all $x\in P$ satisfying the following two properties.
\begin{enumerate}[label={\rm (E\arabic*)}]
\item \label{E1} $|B^x|\leq \frac{\eta^2 t}{10}$.
    \item \label{E2} For each $i\in [\eta^2 t/\ell]$, each set $\alpha(L_x(0,i\ell))$ and $\alpha(L_x(2t-i\ell,2t))$ contains at most $(i-1)$ numbers in $B^x$.
\end{enumerate}
We claim the following.
\begin{claim}\label{cl: P' size}
    $|P'|\geq \eps k/6$.
\end{claim}
\begin{proof}
Let $P^{(0)}$ be the set of $x\in P$ which does not satisfy \ref{E1}.
For each $i\in [\eta^2 t/\ell]$, we let 
$P^{(i)}$ be the set of all $x\in P\setminus \bigcup_{j\in [0,i-1]} P^{(j)}$ that does not satisfy \ref{E2} for the index $i$.

First, by Claim~\ref{cl: B distribution}, we have $|B|\leq \ell^2 \eta^3 k^2$.
By the definition of $P^{(0)}$, the set
$\bigcup_{x\in P^{(0)}} \alpha(V_X)$ contains at least $ \frac{\eta^2t}{10}|P^{(0)}|$ numbers in $B$. Hence we have
$$ \frac{\eta^2t}{10}|P^{(0)}| \leq \ell^2 \eta^3 k^2,$$
implying $|P^{(0)}|\leq \eps^2 k$ from \eqref{eq: s size}.

Now, consider $i\in [\eta^2 t/\ell]$.
For each $x\in \bigcup_{j\in [i]}P^{(j)}$ and $j\in [i]$ the definition of $P^{(j)}$ ensures that $$\bigcup_{x\in P^{(j)}} \alpha(L_x(0,i\ell)) \supseteq \bigcup_{x\in P^{(j)}} \alpha(L_x(0,j\ell))$$
contains at least $j|P^{(j)}|$ numbers in $B$.
Because $P^{(1)},\dots, P^{(i)}$ are pairwise disjoint by definition and $\alpha$ is injective,  $\sum_{j\in [i]}j|P^{(j)}|$ is at most the number of elements of $B$ inside
$\bigcup_{j\in [i]} \bigcup_{x\in P^{(j)}} \alpha(L_x(0,i\ell))$, which is at most 
$\frac{2(i-1)\eta^2 k}{\log k}$ by Claim~\ref{cl: B distribution}.
Hence we obtain the following for all $i\in [\eta^2 t/\ell]$,
$$\sum_{j\in [i]}j|P^{(j)}| \leq \frac{2(i-1)\eta^2 k}{\log k}.$$
By Lemma~\ref{lem: inequality}, this implies
$$ \sum_{i\in [\eta^2 t/\ell]}|P^{(j)}| \leq \frac{4 \eta^2 k}{\log k} \cdot \log(\eta^2 t/\ell) \leq 4\eta^2 k.$$
Therefore, the set
$$P'= P \setminus \bigcup_{i\in[0,\eta^2t]} P^{(i)}$$
has size at least $\eps k/5 - \eps^2 k - 2\eta^2 k \geq \eps k/6$.
\end{proof}

For each $x\in P'$, consider the reduced grid digraph $G_x$ with the set $Z_x:= \alpha^{-1}(B^x)$.
By \ref{E1}--\ref{E2}, for $x\in P'$,
the graph $G_x = G(t;\eta^2 t)$ and the set $Z_x$ satisfy the assumption of  Lemma~\ref{lem: grid} with $\eta^2,t$ playing the roles of $\eta^2, n$, respectively.
Thus there exists a path $P_x$ in $G_x - Z_x$ from $(x,0,0)$ to $(x,t,t)$.
As each vertex $(x,s_1,s_2)$ in $G_x$ corresponds to a number $\alpha(x,s_1,s_2)=x+ s_1m_1(k) + s_2m_2(k)$,
such a path $P_x$ yields
  a sequence $\alpha(P_x)$ of numbers starting from $x$ and ending at $x+ t(m_1(k)+m_2(k))$.
For such a path $P_x = (p_1,\dots, p_{2t})$,
let $\mathbf{r}(x)$ be the sequence 
$$\mathbf{r}(x)= (\alpha(p_1)-x, \alpha(p_2)-\alpha(p_1),\dots, \alpha(p_{2t})-\alpha(p_{2t-1})).$$
Then this sequence naturally yields a new sequence
$$\mathbf{u}(x) = (r_1,\dots, r_{t},w_1,\dots, w_{t'}, r_{t+1},\dots, r_{2t}).$$
As $P_x$ avoids the vertices in $Z_x = \alpha^{-1}(B^x)$, the set of numbers in the sequence
$$x+\mathbf{r}(x) =( x, x+r_1,\dots, x+ \sum_{i\in [2t]} r_{i})$$
does not intersect $B$. 
By Claim~\ref{cl: r useful}, this implies that 
the sequence 
$$x+\mathbf{u}(x) = (x, x+ u_1,\dots, x+ \sum_{i\in [2t+t']} u_i) $$
from $x$ to $y(x)$ does not intersect $B_{i_*}$.
As this sequence starts with $x\in \Phi^{+}$ and ends at $y(x)\in \Phi^{-}$ and the differences between two consecutive terms are all $m_s(k)$ for some $s\in I\cup \{1,2\}\subseteq [k_0-k]$, it contains two consecutive terms which form a bad pair. As this new bad pair belongs to the sequence $x+ \mathbf{u}(x)$, it is disjoint from $B_{i_*}$.
Same logic applies for all $x\in P'$, and we obtain $\eps k/6$ new bad pairs.
As $\alpha$ is injective and each bad pair from the sequence $\mathbf{u}(x)$ belongs to $\alpha(V_x)$, those bad pairs obtained from different choices of $x\in P'$ are pairwise disjoint.

Among them we choose exactly $\eta^2 k /\log{k} \leq \eps k/6 \leq |P'|$ pairwise disjoint bad pairs and add them  to $\mathcal{B}_{i_{*}}$ to obtain $\mathcal{B}_{i_{*}+1}$.
Then, by construction, $\mathcal{B}_{i_{*}+1}$ satisfies \ref{Ind1}$_{i_{*}+1}$.
Moreover, \ref{Ind2}$_{i_{*}}$ together with this construction of new bad pairs, 
$$| \bigcup_{i_{*}-j+1}^{i=i_{*}} (S_i\cup T_i) \cap B_{i_{*}+1}| \leq 
| \bigcup_{i_{*}-j+1}^{i=i_{*}-1} (S_i\cup T_i) \cap B_{i_{*}}| + \frac{\eta^2 k}{\log{k}}.$$
This implies \ref{Ind2}$_{i_*+1}$. Hence, we can keep repeating this until $i_* =\eta k/2$, then we obtain at least $\frac{\eta^3 k^{2}}{2\log{k}}$ pairwise disjoint bad pairs, and Lemma~\ref{lem: bad} implies that there are $ \frac{\eta^3 k^{2}}{2\log{k}}$ distinct monochromatic solutions. This proves Lemma~\ref{lem: many intervals}.

\section{ Proof of Theorem~\ref{thm: existence}}

Again, for a given large number $k$, we write $k_{0}$ to denote the largest number with $p(k_{0})< 2 p(k) - 4m_1(k)$.
Let $m'(k)$ be the smallest integer satisfying $p(m'(k) )\geq 4m_1(k)$, then we have $m'(k)= \Theta(k^{(d-1)/d})$ for sufficiently large $k$.
We first collect the following lemma.

\begin{lemma}\label{lem: initial}
Let $p(z)\in \mathbb{Z}[z]$ be a non-odd polynomial of degree $d\geq 2$ with a positive leading coefficient.
Let $k$ be a sufficiently large number.
A $2$-coloring $\phi: [n,p(k)-n] \rightarrow \{-1,1\}$ is given where $k$ is a positive switch in this coloring.
Assume that it has no monochromatic solutions.
Let $n'=\max\{ n, m'(k)\}$ and let 
$$ I_{-1}= [k+1,k_{0}] \text{ and } I_1= [n', k].$$
Then for each $i\in \{-1,1\}$, all numbers in the interval $I_i$ has color $i$. 
\end{lemma}
\begin{proof}
As $k$ is a positive switch, we have $\phi(k)=+1$ and $\phi(k+1)=-1$. 
Let $m_1:=m_1(k)$ and $m'= m'(k)$. 
As $p(z)$ is a non-odd polynomial, we know that $m_1$ is an odd number.
By the lack of monochromatic solutions, Lemma~\ref{lem: bad} implies that 
\begin{align}\label{eq: monotone s}
\text{for every $a\in [n,p(k)-n-m_1]$, 
$\phi(a) \leq \phi(a + m_1)$. }
\end{align}

We want to apply Proposition~\ref{prop: sum} with $m_1,n',Q$ playing the roles of $m,n,p(k)$, respectively.
For this, condition \ref{A1} is satisfied by \eqref{eq: monotone s} and as $\phi(k)=+1$ and $\phi(k+1)=-1$ and $p(k+1)=p(k)+ m_1$, the lack of monochromatic solutions implies \ref{A2} and \ref{A3}. 
As $k$ is large,  we have $m'< \frac{1}{10} m_1$, the application of Proposition~\ref{prop: sum} deduces that any number in $[p(k)+m_1+1, 2p(k)-4m_1]$ can be written as a sum of two numbers of color $+1$. As $p(k+1)= p(k)+m_1$ and $p(k_0)\leq 2p(k)-4m_1$, all numbers $p(x)$ with $x\in I_{-1}$ can be written as a sum of two numbers of color $1$.
Hence, the lack of monochromatic solutions implies that all numbers in $I_{-1}$ has color $-1$.

%If $\phi(k-1)=-1$, then $k-1$ is a negative switch. Then we repeat the above argument with two colors $+1$ and $-1$ swapped and $k-1$ playing the role of $k$. This proves that $\phi(k+1)=\phi(k)=+1$, which is a contradiction. Hence we have $\phi(k-1)=+1$.

Similarly, Proposition~\ref{prop: sum} deduces that any number in $[4m', p(k)-1]$ can be written as a sum of two numbers of color $-1$. As $p(k-1)< p(k)$ and $p(n') \geq p(m'(k)) \geq 4m_1$, the number $p(y)$ can be written as a sum of two numbers with color $-1$ for all numbers $y$ in $I_{1}$.
Hence the lack of monochromatic solutions yields that all the numbers in $I_{1}$ are colored $+1$.
Therefore, all numbers in $I_{-1}$ are colored $-1$ and all numbers in $I_1$ are colored $+1$ and this proves the lemma.
\end{proof}

Now, for sufficiently large $n$, we choose numbers $q,\ell$ so that $$0<1/n\ll 1/\ell \ll 1/q \ll 1/\|p\|<1.$$
We consider a coloring $\phi: [n,  p(\lceil \frac{p(n)}{2} \rceil)] \rightarrow \{-1,1\}$ and assume that there are no monochromatic solutions.
If the interval $[n, \lceil \frac{p(n)}{2} \rceil]$ is monochromatic, then we obtain a monochromatic solution  $(\lfloor \frac{p(n)}{2} \rfloor , \lceil \frac{p(n)}{2} \rceil, n)$, so we are done.

Otherwise, we choose the largest switch $k\in [n, \lceil \frac{ p(n)}{2} \rceil]$.
Without loss of generality, assume that $k$ is a positive switch, i.e.  $\phi(k)=+1$ and $\phi(k+1)=-1$.
As $k\leq \lceil \frac{ p(n)}{2} \rceil$, we have $p(k) \leq p(\lceil \frac{p(n)}{2} \rceil)$ and we have $p(k_{0}) \leq 2p(\lceil \frac{p(n)}{2} \rceil)- 4m_1(\lceil \frac{p(n)}{2} \rceil)$.
Moreover, Lemma~\ref{lem: initial} implies that $k$ is an isolated switch with all numbers in $[k+1,k_{0}]$ colored $-1$ and all numbers in $[\max\{n, m'(k)\},k]$ are colored $+1$.

Thus, by Lemma~\ref{lem: bad} and the assumption that there is no monochromatic solution, the following holds.
\begin{align}\label{eq: no bad}
\text{If $b=a+m_s(k)$ for some $s\in [k_0-k]$, then $\phi(a)\leq \phi(b)$.}
\end{align}

{\noindent \bf Case 1.} $k\geq n+4\ell$.
In this case, we claim that $[k-4\ell, k]$ is monochromatic with color $+1$.
As $1/n\ll 1/\ell$ and $k\geq n$, we have $p(k-4\ell) \geq \ell^2 m_1(k)$ by \eqref{eq: mk property}.
Hence, $m'(k) \leq k-4\ell$ and  Lemma~\ref{lem: initial} implies that all numbers in $[k-4\ell, k]$ must have color $+1$.

Let $I:= \{t\cdot 2^j : t\in [q], 1\leq j \leq q+ \frac{1}{2} \log{k}\}$.
Let $s$ be the smallest number such that
$$T:= \frac{p(k)}{2}- k +2\ell - \mathbf{m}^k \cdot (q^q \mathbf{1}_I) - s m_1(k)$$
is smaller than $qk^{d-1}$. Such an integer exists as $2m_1(k)< qk^{d-1}$.

By Proposition~\ref{lem: estimation}, there exists a vector $\mathbf{v} + q^q \mathbf{1}_I \in [2q^q]^k$ such that $|\mathbf{v} \cdot \mathbf{m}^{k}| = T  \pm \ell$ and  $\mathbf{v} + q^q\mathbf{1}_I$ has nonzero $i$-th coordinate only when 
$i\in I$. In particular, $i\in I$ implies that $i \in [q^q k^{1/2}]\subseteq  [k_0-k]$.
Consider the vector $\mathbf{x}= \mathbf{v} + q^q \mathbf{1}_I + s\mathbf{1}_{\{1\}}$, and for each $i\in [k_0-k]$, choose $\mathbf{x}_i$ copies of the numbers $m_i(k)$ and list them in an arbitrary way.
Let $w_1,w_2,\dots, w_t$ be the obtained list of numbers, then each $w_i$ is $m_{i}(k)$ for some $i \in [k_{0}-k]$ and 
$$Q:=\sum_{i\in [t]} w_i = \frac{p(k)}{2} - k + 2\ell \pm \ell.$$
Hence, this with \eqref{eq: no bad} implies that all numbers in $[k-4\ell,k]+Q$ must have color $+1$.
On the other hand, by the choice of $Q$, at least $3\ell$ numbers in $[k-4\ell,k)+Q$ lies in the interval $[ \frac{p(k)}{2}- 2\ell, \frac{p(k)}{2} + 2\ell]$.
Hence, more than half of the numbers in the interval $[ \frac{p(k)}{2}- 2\ell, \frac{p(k)}{2} + 2\ell]$ has color $+1$, and this yields a monochromatic solution of the form $(\frac{p(k)}{2}-y, \frac{p(k)}{2}+y, k)$ with color $+1$. \newline

{\noindent \bf Case 2. } $n\leq k<n+4\ell$. In this case, maximality of $k$ implies that $[n+4\ell+1,  \lceil p(n)/2\rceil]$ is monochromatic with color $-1$. Then we have $k+4\ell+1< k + m_1(k) <   \lceil p(n)/2\rceil$, hence
$\phi(k)=1 > \phi(k+ m_1(k) )=-1$, thus this contradicts \eqref{eq: no bad}.  

\medskip

This proves Theorem~\ref{thm: existence}.


\begin{thebibliography}{99}

\bibitem{BDKPP}
Zs.~Baja, D.~Dob\'ak, B.~Kov\'acs, P.~P.~Pach, D.~Pigler,
\newblock Towards characterizing the 2-Ramsey equations of the form $ax+by=p(z)$, Discrete Mathematics 346 (5) (2023) 113324


\bibitem{BL}
V.~Bergelson and A.~Leibman, 
\newblock Polynomial extensions of van Der Warden’s and Szemer\'edi's theorems, J. Am. Math. Soc. 9 (3) (1996) 725--753.

\bibitem{BS}
M.~Bowen and M.~Sabok, 
\newblock
Monochromatic products and sums in the rationals, arXiv:2210.12290.


\bibitem{CCS}
P.~Cameron, J.~Cilleruelo and O.~Serra,
\newblock
On monochromatic solutions of equations in groups, Revista Matemática Iberoamericana 23.1 (2007) 385--395.


\bibitem{CGMS}
M.~Campos, S.~Griffiths, R.~Morris and J. Sahasrabudhe, 
\newblock 
An exponential improvement for diagonal Ramsey, arXiv:2303.09521. 


\bibitem{CGS}
P.~Csikv\'ari, K.~Gyarmati and A.S\'ark\"{o}zy,
\newblock
Density and Ramsey type results on algebraic equations with restricted solution sets, Combinatorica 32 (2012) 425--449.

\bibitem{Dat}
B.~Datskovsky,
\newblock
On the number of monochromatic Schur triples, Advances in Applied Math, 31 (2003), 193--198.



\bibitem{FGR}
P.~Frankl, R.~L.~Graham and V.~R\"odl,
\newblock 
Quantitative Theorems for Regular
Systems of Equations, J. Comb. Theor. Ser. A, 47, (1988) 246--261.

\bibitem{GL}
B.~Green and S.~Lindqvist,
\newblock 
Monochromatic solutions to $x+y=z^2$ Can. J. Math. 71 (3) (2019) 579--605.


\bibitem{HKM}
M.~J.H.~Heule, O.~Kullmann and V.~W.~Marek,
\newblock 
Solving and Verifying the boolean Pythagorean Triples problem via Cube-and-Conquer. In Theory and Applications of Satisfiability Testing – SAT (2016) pp. 228--245.






\bibitem{LPS}
H.~Liu,  P.~P.~Pach,  Cs.~S\'andor,  
\newblock Polynomial Schur's theorem, Combinatorica 42 (2) (2022) 1357--1384.




\bibitem{Moreira}
J.~Moreira,
\newblock
Monochromatic sums and products in $\mathbb{N}$,
Ann. of Math. 185 (2017), 1069--1090.

\bibitem{P}
P.~P.~Pach,
\newblock Monochromatic solutions to $x+y=z^2$ in the interval $[N,cN^4]$. Bull. Lond. Math. Soc. 50 (6) (2018) 1113--1116.

\bibitem{Rado}
R.~Rado, 
\newblock 
Studien zur Kombinatorik, Math. Z. 36 (1933) 424--70



\bibitem{Ramsey}
F.~P.~Ramsey, 
\newblock On a problem of formal logic, Proc. Lond. Math. Soc. 30 (1930) 264--286.


\bibitem{RZ}
A.~Robertson and D.~Zeilberger,
\newblock A 2-coloring of $[1,N]$ can have $(1/22)N^2+O(N)$ monochromatic {S}chur triples, but not less!, Elect. J. Combinatorics, 5 (1998), R19.


\bibitem{Prendiville}
S. Prendiville,
\newblock Counting Monochromatic Solutions to Diagonal Diophantine Equations, Discret. Anal. 2021:14.

\bibitem{Schoen}
T.~Schoen, 
\newblock 
The Number of Monochromatic {S}chur Triples, Europ. J. Combinatorics (1999) 20, 855–-866.

\bibitem{Schur1916}
I.~Schur, 
\newblock 
\"{U}ber die Kongruenz $x^m+y^m\equiv z^m$ (mod $p$), Jahresber. Dtsch. Matah.-Ver. 14 (1916) 114--117.

\bibitem{SV}
O.~Serra and L.~Vena,
\newblock
On the number of monochromatic solutions of integer linear systems on Abelian groups
European J. Combin., 35 (2014) pp. 459--473.


\bibitem{Spencer}
J.~Spencer,
\newblock
Ramsey’s theorem - a new lower bound, J. Comb. Theor. Ser. A 18 (1975) 108–115.


\bibitem{van der Waerden}
B.L.~van der Werden, 
\newblock 
Beweis einer baudeutschen Vermutung, Nieuw. Arch. Wink. 15 (1927) 212--216.




\end{thebibliography}
\end{document}